\documentclass[a4paper,11pt]{article}
\usepackage[latin1]{inputenc}
\usepackage{amsfonts}
\usepackage{amsthm}
\usepackage{amsmath}
\usepackage{amsfonts}
\usepackage{latexsym}
\usepackage{amssymb}
\usepackage{dsfont}
\usepackage[usenames]{color}

\newtheorem{teo}{Theorem}[section]
\newtheorem*{teo*}{Theorem}
\newtheorem{lem}[teo]{Lemma}
\newtheorem{cor}[teo]{Corollary}
\newtheorem{pro}[teo]{Proposition}

\theoremstyle{definition}

\theoremstyle{remark}
\newtheorem{nota}[teo]{Notation}

\newtheorem{rem}[teo]{Remark}

\def\n0{n_{ \text{\rm \tiny o}}}

\def\bce{\begin{center}}
\def\ece{\end{center}}

\def\cO{{\mathcal O}}

\def\cJ{{\mathcal J}}

\def\noi{\noindent}
\def\cF{\mathcal F}
\def\cG{\mathcal G}

\def\bm{\left[\begin{array}}
\def\em{\end{array}\right]}
\def\ben{\begin{enumerate}}
\def\een{\end{enumerate}}
\def\bit{\begin{itemize}}
\def\eit{\end{itemize}}
\def\barr{\begin{array}}
\def\earr{\end{array}}

\def\la{\lambda}

\def\R{\mathbb{R}}
\def\C{\mathbb{C}}

\def\Z{\mathbb{Z}}

\def\cJ{\mathcal{J}}

\def\cC{\mathcal{C}}

\def\cH{\mathcal{H}}
\def\cK{\mathcal{K}}

\def\cP{\mathcal{P}}

\def\cR{{\cal R}}
\def\cS{{\cal S}}
\def\cT{{\cal T}}
\def\cI{{\cal I}}

\def\cB{{\cal B}}

\def\cV{{\cal V}}
\def\cU{{\cal U}}
\def\cW{{\cal W}}

\def\cY{\mathcal{Y}}

\def\fS{\mathfrak{S}}

\DeclareMathOperator{\Tr}{Tr}

\def\beq{\begin{equation}}
\def\eeq{\end{equation}}

\def\CA{\cC_0(A)}
\def\FS{\fS}

\begin{document}
\title{Geometric approach to the Moore-Penrose inverse and the polar decomposition of perturbations by operator  ideals}
\author{Eduardo Chiumiento and Pedro Massey}
\date{}
\maketitle

\begin{abstract}
We study  the Moore-Penrose inverse of perturbations by a symmetrically-normed ideal of  a closed range operator on a Hilbert space. We show that  the notion of essential codimension of projections gives a characterization of  subsets of such perturbations in which the Moore-Penrose inverse is continuous with respect to the metric induced by the operator ideal. These subsets are maximal satisfying the continuity property,  and they carry the structure of real analytic Banach manifolds, which are  acted upon transitively by the Banach-Lie group consisting of invertible operators associated with the ideal.  This geometric construction allows us to prove that  the Moore-Penrose inverse is indeed a real bianalytic map between infinite-dimensional manifolds. We  use these results to study the polar decomposition of closed range operators from a similar geometric perspective.  At this point we prove %as an auxiliary result
 that operator monotone functions are real analytic in the norm of any symmetrically-normed ideal. Finally, we show that the maps defined by the operator modulus  and the polar factor in  the polar decomposition  of closed range operators are real analytic fiber bundles. 
 \end{abstract}

\bigskip

\noi {\bf 2020 MSC:}   Primary 22E65,  58B10  - Secondary 47A55, 47B10

%Primary 
% 22E65 Inf dim Lie groups and their Lie algebras: general properties
%47A55 Perturbation theory of linear operators
%58B10 Differentiability questions for infinite-dimensional manifolds
%Secondary 
%47B10 Linear operators belonging to operator ideals (nuclear, $p$-summing, in the Schatten-von Neumann classes, etc.)
%47A53 (Semi-) Fredholm operators; index theories 

\noi {\bf Keywords:} Moore-Penrose inverse;  polar decomposition; essential codimension;  symmetrically-normed ideal; homogeneous space; real analytic map.

\tableofcontents

\section{Introduction}
  Let $\cH$ be a separable complex infinite-dimensional Hilbert space, and let  $\cB(\cH)$ be the algebra  of bounded linear operators on $\cH$. We denote by   $\cC \cR \subset \cB(\cH)$ the set of closed range operators. Let $\fS$ be a symmetrically-normed ideal on $\cH$ equipped with a norm $\|  \, \cdot \,  \|_\fS$. 
%, which is defined  in terms of singular values and a symmetric norming function $\Phi$. 
For a fixed $A \in \cC \cR$, we consider the set of closed range operators that are perturbations of $A$ by operators in the ideal $\fS$, namely
  $$
  \cC \cR \cap (A+ \fS)=\{ B \in \cC \cR : B- A \in \fS \}.
  $$ 
A natural metric  is defined by   $  d_\fS(B_1 , B_2)=\| B_1 - B_2\|_\fS$, for $B_1 , B_2 \in \cC \cR \cap (A+ \fS)$.  For $X \in \cB(\cH)$, we write $X=V_X |X|$ for its (unique) polar decomposition, where $V_X$ is a partial isometry with the same nullspace as $X$, usually known as the polar factor, and $|X|=(X^*X)^{1/2}$ is the operator modulus.
Denote by $\cC \cR^+$ the set of closed range positive operators and $\cP \cI$ the set of partial isometries on $\cH$.     In the present paper we introduce a geometric framework to study the continuity and real analyticity of  the following maps:
\begin{itemize}
\item[] $\mu: \cC \cR \cap (A+ \fS) \to \cC\cR$, $\mu(B)=B^\dagger$ (Moore-Penrose inverse);
\item[] $\alpha: \cC \cR \cap (A+ \fS) \to \cC \cR^+$, $\alpha(B)=|B|$ (operator modulus);
\item[] $v:\cC \cR \cap (A+ \fS) \to \cP \cI$, $v(B)=V_B$ (polar factor).
\end{itemize}
Our geometric constructions to study these maps  fit into the context of  Banach manifolds related to operator theory; in particular, we deal with  real analytic homogeneous spaces of Banach-Lie groups
 associated with operator ideals (see, e.g., \cite{ALR10, B, BR05, BGJP, BL23, dH72, neeb98,  N00, Lar19}).    The notion of  essential codimension of a pair of projections \cite{BDF73}, or in other words the Fredholm index of a pair of projections \cite{AS94, ASS94}, plays a crucial role  throughout the present work. Remarkably, the essential  codimension appears as a useful tool in a variety of problems in operator theory and geometry such as  Kadison Pythagorean's Theorem \cite{KL17}, equivalence of quasi-free states \cite{SV78},  unitary equivalence of projections \cite{ASS94, S17}, geodesics in the Grassmann manifold and Toeplitz kernels \cite{A14, ACL18} and restricted diagonalization \cite{ECPM1, L19}.

     \medskip

\noi \textit{Previous related results.} The Moore-Penrose  inverse and the polar decomposition  are  ubiquitous  in linear algebra, matrix analysis and operator theory.   
 The continuity and differentiability of the Moore-Penrose inverse have been extensively discussed. The map defined by taking the Moore-Penrose inverse of complex matrices of size $d\geq 1$ is continuous at a matrix $A$ when one restricts its domain to the set of all the matrices with rank equal to $\mathrm{rank}(A)$. 
%Indeed, in the finite-dimensional case, the rank can be used to decompose the set of all matrices.
 Notice that the 
product group $\cG\ell(d)\times \cG\ell(d)$ of invertible matrices of size $d\geq 1$ acts transitively 
on the set of matrices of (constant) rank equal to $\mathrm{rank}(A)$ by $(G,K)\cdot A= GAK^{-1}$. 
Hence, the set of matrices of rank equal to $\mathrm{rank}(A)$ turns out to be a connected set that admits a real analytic manifold structure.   Furthermore,  for any differentiable map of a real parameter taking values in the manifold of matrices with constant rank, the composition of this map with the Moore-Penrose inverse is also differentiable. These results have interesting consequences in the perturbation theory  of matrices (see \cite{GPe73, S77, W73}).

  In the infinite-dimensional case, Labrousse and Mbekhta \cite{LMb92} proved that 
the maps given by the Moore-Penrose inverse and the polar factor are continuous at $A \in \cC \cR \subset \cB(\cH)$ if and only if $A$ is injective or surjective. Other works on the continuity of the Moore-Penrose inverse on Hilbert spaces deal with the reduced modulus minimum, or consider the projections onto the range or nullspace in place of the notion of rank \cite{CAp, CWS96, I83}.  More generally, we refer to \cite{Boa, K01, LR} for extensions of this circle of ideas to Banach algebras. Recently, the stability of the Moore-Penrose invertibility under compact perturbations has been studied in \cite{BJia}.

The theory of Banach-Lie groups and their homogeneous spaces provide an interesting point of view for understanding several objects in operator theory. In this direction we mention three concrete motivations for our work. First, the differential geometry of generalized inverses investigated by Andruchow, Corach and Mbekhta \cite{ACM05}.   Second, the study of several metrics on the set of closed range operators by Corach, Maestripieri and Mbekhta \cite{CMM09}. In particular, they introduced on $\cC \cR$ an action of the product group  of invertible operators $\cG \ell(\cH) \times \cG \ell(\cH)$, and they showed that for any $A \in \cC \cR$ the orbits
\begin{equation}\label{orbit gl}
\cO(A)=\{  GAK^{-1} :  G,K \in \cG \ell(\cH)  \}
\end{equation}
are topological homogeneous spaces. The metrics considered on $\cC \cR$  in this result are given by $d_R(B_1,B_2)=\|B_1 - B_2 \|  + \|P_{R(B_1)}   - P_{R(B_2)}\|$, or $d_N(B_1,B_2)=\|B_1 - B_2 \|  + \|P_{N(B_1)}   - P_{N(B_2)}\|$,  for $B_1, B_2 \in \cC \cR$, where the norm on each term is the operator norm. Here  $P_{N(B_i)}$ and  $P_{R(B_i)}$ denote the projections onto the nullspace and range of $B_i$, respectively. These metrics allow the construction of continuous local cross-sections for the action, defined in terms of the Moore-Penrose inverse. Third, the work on generalized inversion due to   Belti\c t$\breve{\text{a}}$,  Goli${\rm \acute{n}}$ski,  Jakimowicz and Pelletier \cite{BGJP}.  Recall that the invertible group of a Banach algebra is a manifold, which is actually an open set of the algebra, and the inversion map is   complex analytic on this manifold by the holomorphic functional calculus.  So the authors proposed to understand  the Moore-Penrose inverse in Banach algebras as an inversion with some pathologies. This lead them to an  application of  the theory of Banach-Lie groupoids, with particular emphasis on the case of $C^*$-algebras.

\medskip

\noi \textit{The results of this paper.} For a symmetrically-normed ideal $\fS$, we consider the Banach-Lie groups  $\cG \ell_\fS:=\cG \ell(\cH) \cap (I + \fS)$ and  $\cU_\fS:=\cU(\cH)  \cap (I + \fS)$, where $\cG \ell(\cH)$ and $\cU(\cH)$ are the full invertible and unitary groups, respectively.
 We denote by $[P:Q]$ the essential codimension of two orthogonal projections $P$ and $Q$. 
In this context, we consider perturbations of  closed range operators $\cC \cR$,  positive closed range operators $\cC \cR^+$ and partial isometries $\cP \cI$ by the symmetrically-normed  ideal $\fS$. For  fixed operators $A \in \cC \cR$, $C \in \cC \cR^+$ and $V \in  \cP \cI$, we will show that their set of perturbations can be decomposed as the following disjoint unions $\cC \cR \cap (A+ \fS)=\cup_{k \in \mathbb{J}_A} \cC_k(A)$,  $\cC \cR^+ \cap (C+ \fS)=\cup_{k \in \mathbb{J}_C} \cP_k(C)$ and 
$\cP \cI  \cap (V+ \fS)=\cup_{k \in \mathbb{J}_V} \cV_k(V)$, where the sets on the unions are defined using the essential codimension as
\begin{align*}
 \cC_k(A) &  := \{  B \in \cC \cR \cap (A+ \fS) :\, [P_{N(B)}: P_{N(A)}] =k \} ;\\
 \cP_k(C) &:=\{  D \in \cC \cR^+ \cap (C+ \fS) :\, [P_{N(D)}: P_{N(C)}] =k \} ;  \\
\cV_k(V) & := \{  X \in \cP \cI \cap (V+ \fS) :\, [P_{N(X)}: P_{N(V)}] =k \}.
\end{align*}
The set of indices $ \mathbb{J}_{A}$, $ \mathbb{J}_{C}$ and $ \mathbb{J}_{V}$ in the previous unions are always infinite subsets of $\Z$, and depend on the dimension of the nullspace,  range and corange of the operators $A$, $C$ and $V$. 

The main results of this paper are the following: 
\begin{itemize}
\item The sets $\cC_k(A)$ and $\cP_k(C)$ are Banach manifolds. Indeed, they admit the structure of real analytic homogeneous spaces, which are also submanifolds of natural affine spaces
 (Theorems \ref{teo ck es esp hom} and \ref{estru pk real analitica}).
\item The map $\mu: \cC_k(A) \to \cC_k(A^\dagger)$, $\mu(B)=B^\dagger$, is a real bianalytic map between Banach manifolds (Theorem \ref{teo pseudo in real analytic}).
\item The maps $\alpha: \cC_k(A)\to \cP_k(|A|)$, $\alpha(B)=|B|$, and  $v:\cC_k(A) \to \cV_k(V_A)$, $v(B)=V_B$, are real analytic fiber bundles between Banach manifolds  (Theorem \ref{teo hay fibras}). 
\end{itemize}
Before proving these results, which hold for every $k \in \mathbb{J}_A$, we show that these maps are (well-defined and) continuous. Indeed, the choice of the sets of the form $\cC_k(A)$ is not arbitrary, 
 these are dense connected subsets of $\cC \cR \cap (A + \fS)$, which are  maximal with respect to the continuity property of the Moore-Penrose inverse (Theorems \ref{pseudo conv} and \ref{index decomp connected}). The sets $\cC_k(A)$ can be roughly described as formed by those closed range operators that are perturnations of $A$ be elements in $\fS$ and have `the same rank with respect to $A$'; these are {\it local} conditions induced by $A$ and $\fS$ (as opposed to the condition of merely having a fixed rank of fixed nullity).  To the best of our knowledge,  the above results on these three maps are also new in the context of  finite-dimensional manifolds. 

The Banach manifold structures of the sets $\cC_k(A)$, $\cP_k(C)$ and $\cV_k(V)$ are not evident from their definitions.
%, so this is  another  relevant issue we address.
   Each $\cC_k(A)$ admits a transitive action of $\cG \ell_\fS \times \cG \ell_\fS$ which consists of restricted versions of the orbits in \eqref{orbit gl}. In particular, $\cC_0(A)$ (the set containing $A$)  has the following characterization:
\begin{equation}\label{glj orbit}
\cC_0(A)=\{  GAK^{-1} :  G,K \in \cG \ell_\fS  \}.
\end{equation}
 This fact is related to our previous work on restricted orbits of closed range operators \cite{ECPM23}. 
Similar results for unitary orbits, associated to operator ideals, of normal operators have been recently obtained in \cite{BL23}.
In contrast to the larger orbits in \eqref{orbit gl}, there is no need to introduce metrics such as $d_R$ or $d_N$ to construct continuous local cross sections for the map $\pi_0: \cG \ell_\fS \times \cG  \ell_\fS \to \cC_0(A)$, $\pi_0(G,K)=GAK^{-1}$. Once these sections are constructed, we can further endow  an orbit like \eqref{glj orbit} with the structure of real analytic homogeneous space that is also a submanifold of $A+\fS$.
% (Theorem  \ref{teo ck es esp hom}).  
On the other hand, the Banach manifold structure of the sets $\cP_k(C)$ is given in terms of congruence orbits of the group $\cG \ell_\fS$ (Theorem  \ref{estru pk real analitica}). For instance when $k=0$, we obtain
\begin{equation}\label{cong orb rest}
\cP_0(C)=\{  GC G^* :  G \in \cG \ell_\fS  \}.
\end{equation}
The motivation for considering the restricted orbits in \eqref{glj orbit} and \eqref{cong orb rest} comes from previous work on partial isometries \cite{EC10, EC19}, where the above defined sets $\cV_k(V)$, $k \in \mathbb{J}_V$, were proved to be orbits of the product group $\cU_\fS \times \cU_\fS$.   
Finally,  we observe that  the real analyticity of the operator modulus 
depends on the real analyticity of the square root on the sets $\cP_k(C)$. We present a more general statement for operator monotone functions in Corollary \ref{teo sobre AvH2}. This follows from Theorem \ref{teo sobre AvH}, which in turn is based on earlier work of Ando and van Hemmen on perturbations by symmetrically-normed ideals \cite{AvH}.

The paper is organized as follows. Section \ref{sec prelis} is devoted to notation and preliminary results.   In Section  \ref{Moore-Penrose inv}  we  introduce several geometric structures and prove the main results on the Moore-Penrose inverse.  In Section \ref{polar decomposition} we establish the results on operator monotone functions, the operator modulus and  the polar factor.

\section{Preliminaries}\label{sec prelis}

Let $\cH$ be an infinite-dimensional (complex separable) Hilbert space, and let $\cB(\cH)$ be the algebra of bounded operators on $\cH$. Given $A \in \cB(\cH)$ we write $N(A)$ and $R(A)$ for the nullspace and range of $A$, respectively. 
The orthogonal projection onto a closed subspace $\cS$ is denoted by $P_\cS$.  

\medskip

\noi \textit{Moore-Penrose inverse and polar decomposition.} The set of all closed range operators on $\cH$ is given by 
$$
\cC\cR=\{ A \in \cB(\cH)\, : \, R(A) \text{ is a closed subspace} \}.
$$
An operator $B \in \cB(\cH)$ is said to be   the  \textit{Moore-Penrose inverse} of $A \in \cB(\cH)$ if it satisfies that $ABA=A$, $BAB=B$, $(AB)^*=AB$ and $(BA)^*=BA$. 
If  the Moore-Penrose exists, then it is uniquely determined, and we denote it by $B=A^\dagger$. It is not difficult to check that $A \in \cB(\cH)$ admits a Moore-Penrose inverse 
if and only if $A \in \cC \cR$. Furthermore, $AA^\dagger=P_{R(A)}$ and $A^\dagger A=P_{N(A)^\perp}$, whenever $A \in \cC \cR$.
The following useful identity was proved by Wedin \cite{W73}:
\begin{equation}\label{stewart identity}
A^\dagger - B^\dagger=-A^\dagger(A-B)B^\dagger + (A^* A)^\dagger (A^* - B^*)(I - BB^\dagger) + (I-A^\dagger A)(A^* - B^*) (BB^*)^\dagger  \,.
\end{equation}  
For $A \in \cB(\cH)$, $A\neq 0$, the reduced minimum modulus  is  given by $\gamma(A)=\min_{\lambda \in \sigma(|A|)\setminus \{ 0\}} \lambda$, where $\sigma(|A|)$ is the spectrum of $|A|$. Equivalently, $\gamma(A)=\inf\{ \|Af\| :  f \in N(A)^\perp \, , \, \|f\|=1\}$.  Recall that for $A \in \cB(\cH)$, we have  $A \in \cC \cR$ if and only if $\gamma(A)>0$. In such case,   $\gamma(A)=\|A^\dagger\|^{-1}$, where 
$\| \, \cdot \, \|$ denotes the operator norm. Also, it holds 
$\gamma(A)=\gamma(|A|)=\gamma(|A^*|)=\gamma(A^*)$, which in particular gives that $A^*$, $|A|$ and $|A^*|$ have closed range if $A$ has closed range. 

 An operator $X \in \cB(\cH)$  is  a  partial isometry if $\|X f \|=\|f\|$, for all $f \in N(X)^\perp$. This is equivalent to having that $XX^*$ is an (orthogonal) projection, or $X^*X$ is an (orthogonal) projection. We write 
 $$
 \cP \cI= \{  X \in \cB(\cH) :   X \text{ is  partial isometry} \}.
 $$ 
 The \textit{polar decomposition} of an operator $A \in \cB(\cH)$ is the factorization $A=V_A |A|$, where $|A|=(A^*A)^{1/2}$ is the \textit{operator modulus} and $V_A\in \cP\cI$ is the unique partial isometry that further satisfies the condition $N(V_A)=N(A)$. In the case where $A \in \cC \cR$, we observe that $|A|\in\cC\cR$ is such that $V_A=A|A|^\dagger$, and we call $V_A$ the \textit{polar factor}.

\medskip

\noi 
\textit{Symmetrically-normed ideals.} We follow the classical book \cite{GK60} (see \cite{Dyk04, S79}).
 A {\it symmetrically-normed ideal} is a two-sided ideal $\fS\subseteq \cB(\cH)$  
endowed with a norm $\|\,\cdot\, \|_\fS$ satisfying $\|ABC\|_\fS \leq \|A\|\| B\|_\fS \|C\|$, for all $A,C \in \cB(\cH)$ and $B \in \fS$; and $\|B\|_\fS=\|B\|$, for every rank-one operator $B$. We also assume that $(\fS,\, \|\,\cdot\, \|_\fS)$ is a Banach space. 
The previous facts imply that $ \|B \| \leq \|B\|_\fS$, for all $B \in \fS$ and that 
$\| B_1 B_2  \|_\fS \leq \|B_1 \|_\fS \|B_2 \|_\fS$, for $B_1 , B_2 \in \fS$, i.e. the norm of the ideal is submultiplicative. 
%For the notion of symmetrically-normed ideal and the notation we use here,  we refer to the classical book \cite{GK60} (see also  \cite{Dyk04, %S79}). In particular we write $(\fS,\, \|\ \|_\fS)$ for a symmetrically-normed ideal %associated with a symmetric norming function $\Phi=\Phi_\fS$ 
%on the Hilbert space $\cH$. 
Recall that $\cF \subseteq \fS \subseteq \cK$, for every symmetrically-normed ideal, where $\cF=\cF(\cH)$ is the ideal of finite-rank operators and $\cK=\cK(\cH)$ is the ideal of compact operators. 
%Each symmetrically-normed ideal $\fS$ is endowed with a norm defined by
%$$
%\|A\|_\Phi=\sup_{n \geq 1} \Phi(s_1(A), \ldots , s_n(A), 0, 0, \ldots), \, \, \, \,  A\in \fS .
%$$
%Here $s(A)=\{ s_n(A)\}_{n \geq 1}$ is the sequence of singular values of $A$ arranged in non-increasing order and counting multiplicities. 
The $p$-Schatten $\fS_p$ ($1 \leq p \leq \infty$) are well-known examples of symmetrically-normed ideals, whose  norms are given by $\|A\|_p=\Tr(|A|^p)^{1/p}=(\sum_{n \geq 1}s_n^p(A))^{1/p}$, $p\geq 1$; and for $p=\infty$, $\fS_\infty=\cK$ endowed with the usual operator norm $\|A\|_\infty=\|A\|=s_1(A)$.
Here $s(A)=\{ s_n(A)\}_{n \geq 1}$ is the sequence of singular values of $A$ arranged in non-increasing order and counting multiplicities. Other examples of symmetrically-normed ideals can be found in the aforementioned references.
 %Often we will use the following properties which relate the norm of the ideal and the operator norm: $\|ABC\|_\fS \leq \|A\|\| B\|_\fS \|C\|$, for %all $A,C \in \cB(\cH)$ and $B \in \fS$; and $ \|B \| \leq \|B\|_\fS$, for all $B \in \fS$.
\medskip

\noi \textit{Essential codimension.} Next we recall the notion of essential codimension (see \cite{AS94, ASS94, BDF73}). Let $P,Q \in \cB(\cH)$ be two orthogonal projections such that  the operator $QP|_{R(P)}:R(P)\to R(Q)$ is Fredholm. In this case,  $(P,Q)$ is known as a Fredholm pair and the index of this Fredholm operator  
\begin{align*}
[P:Q]& :=\mathrm{Ind}(QP|_{R(P)}:R(P)\to R(Q)) \label{index pair}\\
& = \dim(N(Q)\cap R(P)) - \dim(R(Q)\cap N(P)) \nonumber
\end{align*}
is called the \textit{essential codimension} (or \textit{Fredholm index of the pair}). We will often have two projections such that  $P- Q \in \cK$. In such a case,  it is easy to see that $(P,Q)$ is a Fredholm pair, and the essential codimension is well-defined. Among some elementary properties of the essential codimension that we will use frequently are the following:  $[P:Q]=-[Q:P]$; if $(P_i,Q_i)$, $i=1,2$, are two Fredholm pairs such that  $P_1P_2=0$ and $Q_1Q_2=0$, then $(P_1+ P_2,Q_1+Q_2)$ is a Fredholm pair and
$[P_1 + P_2:Q_1 + Q_2]=[P_1:Q_1] + [P_2:Q_2]$; and if $(P,Q)$ and $(Q,R)$ are Fredholm pairs, 
and either $Q-R \in \cK$ or $P-Q \in \cK$, then $(P,R)$ is a Fredholm pair and $[P:R]=[P:Q] + [Q:R]$.

\medskip

\noi \textit{Banach manifolds.} We consider real analytic manifolds modeled on Banach spaces (see \cite{B, Up85}). Given $M$, $N$ manifolds and a real analytic map $f:M \to N$, we denote
by $T_p f:(TM)_p \to (TN)_{f(p)}$ the tangent map at $p \in M$, where $(TM)_p$  and  $(TN)_{f(p)}$ are the tangent spaces of $M$ at $p$ and $N$ at $f(p)$. A bijective map $f:M \to N$  is real bianalytic if $f$ and $f^{-1}$ are real analytic.
A real analytic map $f:M \to N$ is called a submersion at $p \in M$ if $N(T_p f)$ is a closed complemented subspace of $(TM)_p$ and $T_p f$ is surjective. If $f:M \to N$ is a submersion at every point $p \in M$, then $f$ is called a submersion. A real Banach-Lie group is a real analytic Banach manifold $G$ such that the group multiplication $G \times G \to G$, $(g,h) \mapsto gh$, and the inverse $G \to G$, $g \mapsto g^{-1}$, are real analytic maps. The construction of the Lie algebra $\mathfrak{g}\simeq (TG)_1$ and  the exponential map $\exp_G: \mathfrak{g} \to G$ of a Banach-Lie group $G$ can be carried out similarly to the case of finite-dimensional Lie groups. Also the exponential map $\exp_G: \mathfrak{g} \to G$ is  a local bianalytic map.  An action of a Banach-Lie group $G$ on a manifold $M$ is a map $L:G \times M \to M$, $L(g, p)=g \cdot p$, $g \in G$ and $p \in M$, such that $h \cdot (g \cdot p)=(hg) \cdot p$ and $1 \cdot p=p$, for all $h,g \in G$ and $p \in M$.  The action is said to be real analytic if the map $L$ is real analytic. A \textit{real analytic homogeneous space of a Banach-Lie group $G$} is a manifold $M$ such that $G$ acts transitively and analytically on $M$, and there exists $p \in M$ such  that the map $\pi_p:G \to M$, $\pi_p(g)=g \cdot p$, is a submersion.

Let $M$ be a manifold, and $N \subseteq M$. A chart $(\phi, \cV, E)$  at $p \in M$ consists in an open neighborhood $\cV$ of $p$, a Banach space $E$ and a homeomorphism $\phi: \cV \to \phi(\cV)\subseteq E$. If for every $p \in N$ there exists a chart $(\phi, \cV, E)$ at $p$, and a closed subspace $F$ complemented in $E$ satisfying
$\phi(\cV \cap N)=F \cap \phi(\cV)$, then $N$ is called a \textit{submanifold} of $M$. In this case, $N$ turns out to be a manifold endowed with the topology inherited from $M$. If $H$ is a subgroup of a Banach-Lie group $G$, then $H$ is said to be a \textit{Banach-Lie subgroup} of $G$ when $H$ is a submanifold of $G$.    

Let $M$ and $N$ be two manifolds.  A \textit{real analytic fiber bundle} is a  real analytic surjective map $f:M \to N$ such that for every $p \in N$ then $f^{-1}(p)$ is a manifold, and there exists an open neighborhood $\cV$ of $p$ and a real bianalytic map $\Psi: f^{-1}(\cV) \to \cV \times f^{-1}(p)$ such that $\pi_1 \circ \Psi=f$, where $\pi_1: \cV \times f^{-1}(p) \to \cV$ is the canonical projection. In such case, 
$f$ turns out to be a submersion.

\medskip

 \medskip

\noi \textit{Restricted groups}.  Let $\cG \ell(\cH)$ be the group of invertible operators on $\cH$. For each symmetrically-normed ideal $\fS$ there is associated the following group
$$
\cG \ell_{\fS}:=\{  G \in \cG \ell(\cH) : G- I \in \fS \}.
$$
Also each symmetrically-normed ideal $\fS$ gives raise to a subgroup of the full unitary group $\cU(\cH)$ defined by
$$
\cU_{\fS}:=\{  U \in \cU(\cH) : U - I \in \fS \}.
$$
For a standard reference  for   these groups in the case of the $p$-Schatten ideals see  \cite{dH72}, meanwhile for the case of general symmetrically-normed ideals see \cite{B}.

\begin{rem}\label{restricted groups}
We collect here several properties of the groups defined above. In what follows we let $\fS$ denote a symmetrically-normed operator ideal.
%$\Phi$ denote the symmetric norming function associated with the ideal $\fS$.

\medskip

\noi $i)$ If $P$, $Q$ are orthogonal projections, then there is a unitary operator $U \in \cU_{\fS}$
such that $Q=UPU^*$ if and only if $P-Q \in \fS$ and $[P:Q] = 0$  (see  \cite[Prop. 3.6]{C85}, or more generally, \cite[Prop 2.3]{L19}).

\medskip

%\noi $ii)$ Let $P$, $Q_n$ be orthogonal projections such that $P-Q_n\in \fS$, for $n\geq 1$, and $\|P-Q_n\|_\Phi\rightarrow 0$. Then, there exist $U_n\in\cG\ell_\fS$ such that $U_nQ_nU_n=P$, for $n\geq 1$, and such that $\|U_n-I\|_\Phi\rightarrow 0$; Indeed, this is a consequence of 
%\cite[Proposition 2.2.]{AL08} for the ideal $\fS_2$ of Hilbert-Schmidt operators; there it is shown the existence of continuous local cross sections of the map $\cU_{\fS_2}\ni U\mapsto U^*PU\subset P+(\fS_2)_{sa}$ around $I\in\cU_{\fS_2}$ and $P\in P+(\fS_2)_{sa}$, where $\cU_{\fS_2}$ and $P+(\fS_2)_{sa}$ are endowed with the Hilbert-Schmidt  metric $d_{2}(C,D)=\|C-D\|_{2}$. The general case of a symmetrically normed operator ideal follows with a straightforward adaption of the proof of the previous result (using item $iii)$ below).

\noi $ii)$  $\cG \ell_{\fS}$ is a real Banach-Lie group endowed with the metric $d_\fS(G,K)=\|G-K\|_\fS$,  for $G, K \in \cG \ell_{\fS}$, whose  Lie algebra is  $\fS$.  Next, consider the unitalization 
$\tilde{\fS}=\{ X + \lambda I : X \in \fS,  \, \lambda \in  \C\}\simeq \fS \oplus \C $.  Each element $Z \in \tilde{\fS}$ is written as $Z=X + \lambda I$, for uniquely determined $X \in \fS$ and $\lambda \in \C$, and $\tilde{\fS}$ is equipped with the norm $\|X + \lambda I\|_{\tilde{\fS}}:=  \|X\|_\fS + |\lambda|$. In this case, $\tilde{\fS}$ is a unital Banach algebra. 
We can embed into $\cG \ell_\fS$ in 
$\tilde{\fS}$ by the identification $T\mapsto (T-I)  + I \in \tilde{\fS}$.  Then $\cG \ell_\fS$ is a Lie subgroup of the Banach-Lie group of invertible elements of $\tilde{\fS}$ having real codimension 2. 
On the other hand, $\cU_\fS$ is a Banach-Lie subgroup of $\cG \ell_\fS$, whose Lie algebra is $\fS_{ah}=\{ X \in \fS  :  X^*=-X \}$, the anti-hermitian operators in $\fS$  (\cite[Prop 9.28]{B}).
The exponential maps of these Lie groups are
given by   $\exp_{\cG \ell_{\fS}}: \fS \to  \cG \ell_{\fS}$, $\exp_{\cG \ell_{\fS}}(X)=e^X=\sum_{n \geq 0}\frac{X^n}{n!}$ and  $\exp_{\cU_{\fS}}=\exp_{\cG \ell_{\fS}}|_{\fS_{ah}}$.

\medskip

\noi $iii)$ The exponential map of $\cG \ell_\fS$ is surjective. We give a proof since we do not find references to this fact.  For  $G \in \cG \ell_\fS$, $G-I \in \fS \subseteq \cK$ yields that $\sigma(G)$ is a  countable set having $1$ as its unique limit point.  Thus, there is a ray $L$  from the origin such that $\sigma(G) \subseteq \C \setminus L$, and the analytic logarithm $\log: \C \setminus L \to \{ z :  \theta - 2 \pi < \arg(z) \leq \theta   \}$ is well-defined, where $\theta \in [0,2 \pi)$ is defined by the ray $L$.  Denote by $\sigma_{\tilde{\fS}}(G)$ the spectrum of $G$ in the Banach algebra $\tilde{\fS}$. Observe that  $\sigma_{\tilde{\fS}}(G)=\sigma(G)$, so we can use the analytic functional calculus in $\tilde{\fS}$ to get $e^{\log(G)}=G$. That is, $\log(G)=X+\lambda I \in \tilde{\fS}$, $X \in \fS$ and $\lambda \in \C$, satisfies $G=e^{X + \lambda I}=e^X e^\lambda$. But $G=(e^{X} -I) e^{\lambda} + e^{\lambda}I$, with $e^X-I\in\fS$. Hence, the uniqueness of writing in $\tilde{\fS}$ gives $e^\lambda=1$, so $G=e^X$. 
\end{rem}

\section{Moore-Penrose inverse}\label{Moore-Penrose inv}

We first study the continuity of  the Moore-Penrose inverse.  Then we prove that 
the maximal sets in which it  is continuous admit the structure of Banach manifolds. This is achieved by using the theory of Banach-Lie groups and their homogeneous spaces. We conclude 
that the Moore-Penrose inverse is a real bianalytic map between Banach manifolds.

%\subsection{Continuous perturbations}\label{sec cont pert}
\subsection{Continuity of the Moore-Penrose inverse}\label{sec cont pert}

We recall an  estimate for matrices with equal rank and some elementary facts obtained in \cite{ECPM23}.

\begin{lem}[\cite{W73}]\label{wedin lema}
Suppose that $A,B$ are matrices such that $\mathrm{rank}(A)=\mathrm{rank}(B)$ and $\| A-B\| < \|A^\dagger\|^{-1}$, then
$$
\|B^\dagger\| \leq \frac{\|A^\dagger\|}{1-\|A^\dagger\|\|A-B\|}.
$$
\end{lem}

\begin{lem}[\cite{ECPM23}]\label{ab y moore penrose} Let $\fS$ be a symmetrically-normed ideal
and take  $A,B \in \cC\cR$ be such that $A-B \in \fS$. Then $A^\dagger - B^\dagger \in \fS$,    $P_{R(A)}-P_{R(B)} \in\fS$ and $P_{N(A)}-P_{N(B)} \in \fS$.
\end{lem}
\begin{proof}
We include a short  proof of this result for the convenience of the reader. From Wedin's formula in Eq.  \eqref{stewart identity} we get $A^\dagger - B^\dagger \in \fS$. The other assertions follow by using that $P_{N(A)^\perp}=A^\dagger A$ and 
$P_{R(A)}=AA^\dagger$.
\end{proof}

We now present a generalization of Wedin's estimate in Lemma \ref{wedin lema} in terms of the essential codimension.

\begin{pro}\label{indexbound}
Consider operators $A,B \in \cC\cR$ satisfying $\|A-B\| <\|A^\dagger\|^{-1}$, $A-B \in \cK$ and  $[P_{N(A)} : P_{N(B)}]=0$. Then,
$$
\|B^\dagger\| \leq \frac{\|A^\dagger\|}{1-\|A^\dagger\|\|A-B\|}.
$$
\end{pro}
\begin{proof}
Since $A^*- B^* \in \cK$, then $P_{N(A)^\perp} - P_{N(B)^\perp } \in \cK$ by Lemma \ref{ab y moore penrose}. From Remark \ref{restricted groups} $i)$ applied to the ideal of compact operators, we know that there exists a unitary $L \in \cU_\cK$ such that $LP_{N(A)^\perp}L^*=P_{N(B)^\perp}$. Next pick $\{ E_n \}_{n \geq 1}$ a sequence of finite-rank projections such that $E_n \leq P_{N(A)^\perp}$ and $E_n \nearrow P_{N(A)^\perp}$ strongly. We set $F_n=L E_n L^*$, $B_n=BF_n$ and $A_n=AE_n$, for all $n  \geq 1$. Notice that $\mathrm{rank}(A_n)=\mathrm{rank}(B_n)$, and also $A_n - B_n \to A-B$ strongly.  Further, we observe that $$A_n- B_n=AE_n -BLE_nL^*=(A-B)E_n- B(L-I)E_nL^* - BE_n(L^*-I)\, ,$$ where each term is multiplied by a compact operator. Thus, we get $\| A_n - B_n -(A-B)\| \to 0$ by a well-known result (see, e.g. \cite[Thm. 6.3]{GK60}). On the other hand, since $A_n^* A_n=E_n A^*A E_n \geq \gamma(A) E_n P_{N(A)^\perp} E_n= \gamma(A) E_n$ and $N(A_n)^\perp=R(E_n)$, then $ \|A_n^\dagger\|^{-1}=\gamma(A_n) \geq  \gamma(A)=\|A^\dagger\|^{-1}$. 

Therefore for large $n$, $\|A_n - B_n \| < \| A^\dagger \|^{-1} \leq \|A_n^\dagger\|^{-1}$, so  we can apply Lemma \ref{wedin lema} to obtain
\begin{equation}\label{boundn}
\|B_n^\dagger\| \leq \frac{\|A_n^\dagger\|}{1-\|A_n^\dagger\|\|A_n-B_n\|} \leq  \frac{\|A^\dagger\|}{1-\|A^\dagger\| \|A_n-B_n\|}.
\end{equation}
We claim that $B_n^\dagger$ converges strongly to $B^\dagger$. This follows by using the formula in Eq. \eqref{stewart identity}, which implies that for $f \in \cH$ one has
\begin{align*}
\|(B_n^\dagger - B^\dagger)f\| & \leq \|B_n^\dagger\| \|(B_n - B) B^\dagger f\| + \|(B_n^*B_n)^\dagger\| \| (B_n^* - B^*) (I-BB^\dagger) f\| + \\
& + \|I- B_n^\dagger B_n\| \|  (B_n^* - B^*) (BB^*)^\dagger f\|.
\end{align*}
Here note that $\|(B_n^* B_n)^\dagger\|=\|B_n^\dagger (B_n^\dagger)^*\|= \|B_n^\dagger \|^2 \leq \|B^\dagger \|^2$, for all $n$, by a similar argument as before with $A_n$ and $A$. Also observe that $\|I- B_n^\dagger  B_n\| =1$ and $B_n^*=F_nB^*$ converges strongly to $B^*$ since $F_n \nearrow P_{N(B)^\perp}=P_{R(B^*)}$. This proves our claim. 

Consider $\epsilon >0$, and take a vector $f \in \cH$, $\|f\|=1$, such that $\|B^\dagger\| \leq \|B^\dagger f\| + \epsilon$. Using that $B_n^\dagger$ converges strongly to $B^\dagger$ we have $\|B^\dagger f\| \leq \|B_n^\dagger f\| + \epsilon \leq \|B_n^\dagger \| + \epsilon$ for all $n$ large enough. This gives
\begin{align*}
\|B^\dagger\|  \leq \|B_n ^\dagger \| + 2\epsilon \leq    \frac{\|A^\dagger\|}{1-\|A^\dagger\| \|A_n-B_n\|} + 2\epsilon.
\end{align*}
Letting $n \to \infty$ and noting that  $\epsilon>0$ is arbitrary, we find the desired estimate. 
\end{proof}

Now we can give our main result on the continuity of the Moore-Penrose inverse.

\begin{teo}\label{pseudo conv}
Let $\fS$ be a symmetrically-normed ideal.
Let $\{B_n\}_{n \geq 1}$ be a sequence in $\cC \cR$ such that $B_n- B \in \fS$ and  $\| B_n - B\|_\fS \to 0$, for some $B \in \cC \cR $. The following conditions are equivalent:
\begin{enumerate}
\item[i)] $[P_{N(B_n)} : P_{N(B)}]=0$ for all sufficiently large $n$; 
\item[ii)] $\sup_n \|B_n^\dagger\| < \infty$;
\item[iii)] $\|  B_n^\dagger - B^\dagger\|_\fS \to 0$;
\item[iv)] $\| P_{N(B_n)} -   P_{N(B)} \|_\fS < 1$ for all sufficiently large $n$;
\item[v)]  $\| P_{N(B_n)} -   P_{N(B)} \| < 1$ for all sufficiently large $n$;
\item[vi)] $N(B_n)^\perp \cap N(B)=\{ 0 \}$ for all sufficiently large $n$. 
\end{enumerate}
\end{teo}
\begin{proof} 
$i) \rightarrow ii)$ First, notice that Lemma \ref{ab y moore penrose} implies that $P_{N(B)}-P_{N(B_n)}\in\fS\subset \cK$, so that the essential codimension in the statement above is well defined. Suppose that   $[P_{N(B_n)} : P_{N(B)}]=0$ for large $n$. Since $\|B_n - B\| \leq \| B_n - B\|_\fS \to 0$, we derive from Proposition \ref{indexbound} that
$\|B_n^\dagger\| \leq \frac{\|B^\dagger\|}{1-\|B^\dagger\|\|B-B_n\|}$ for sufficiently large $n$. 
Hence $\sup_n \|B_n^\dagger \| < \infty$. 

\medskip

\noi $ii) \to iii)$ The Moore-Penrose inverse of an operator $A \in \cC \cR$ satisfies $(A^*A)^\dagger=A^\dagger (A^*)^\dagger$ and $(A^*)^\dagger=(A^\dagger)^*$.  Using these facts in the identity in Eq. \eqref{stewart identity}, we get
\begin{align}
\|B_n^\dagger - B^\dagger \|_\fS & \leq \|B_n\| \|B_n-B\|_\fS \| B\| +  \|B_n^\dagger \|^2 \|B_n - B \|_\fS  \|I- BB^\dagger\|  \nonumber\\
&  + \|I-B_n^\dagger B_n \| \|B_n  - B \|_\fS \| B^\dagger\|^2. \label{bound proof cont}
\end{align}
  Notice that $ \|I-B_n^\dagger B_n \|=\|  P_{N(B_n)} \| =1$, and $\sup_n \| B_n \| < \infty$ because $\|B_n - B\| \leq \|B_n - B\|_\fS \to 0$. 
Thus, the assumption $\sup_n \|B_n^\dagger \| < \infty$ implies that $\|B_n^\dagger - B^\dagger \|_\fS \to 0$.

\medskip

\noi $iii) \to iv)$  Recall that $P_{N(B_n)}=I-B_n^\dagger B_n$ and $P_{N(B)}=I-B^\dagger B$. Then 
$$
\| P_{N(B_n)} - P_{N(B)}\|_\fS\leq \|(B_n^\dagger - B^\dagger) B_n  \|_\fS+\|B^\dagger (B_n-B)\|_\fS\,.
$$Thus, $\| P_{N(B_n)} - P_{N(B)}\|_\fS$ becomes arbitrarily small for sufficiently large $n$ because  $\|B_n - B \|_\fS \to 0$ so $\sup_n \|B_n\|_\fS<\infty$, $\|B_n^\dagger - B^\dagger \|_\fS \to 0$ and the norm of the ideal is submultiplicative.

\medskip

\noi $iv) \to v)$  This follows again by the estimate $\| P_{N(B_n)} - P_{N(B)}\|  \leq \| P_{N(B_n)} - P_{N(B)}\|_\fS$.

\medskip

\noi $v) \to vi)$ Straightforward. 

\medskip

\noi $vi) \to i)$ Notice that $\|B_n - B\| \leq \| B_n - B\|_\fS \to 0$, so it holds $N(B)^\perp \cap N(B_n)=\{ 0\}$ for large $n$. Indeed, if there is a unit vector $f_n \in N(B)^\perp \cap N(B_n)$ for infinitely many $n \geq 1$, we find that $0 < \gamma(B) \leq \|B f_n \| =\| (B_n - B)f_n\| \to 0$, a contradiction.
 Hence $[P_{N(B_n)} : P_{N(B)}]= 0$ for all sufficiently large $n$.
\end{proof}

\begin{rem}
 We can take the operator adjoint  and use elementary properties of the essential codimension to state other equivalent conditions. Indeed, we can replace conditions $i)$, $iv)$, $v)$ and $vi)$ by:
 $i')$ $[P_{R(B)}: P_{R(B_n)}]=0$;
$iv')$ $\| P_{R(B_n)} - P_{R(B)} \|_\fS<1$; 
$v')$ $\| P_{R(B_n)} - P_{R(B)} \|<1$; and 
$vi')$ $R(B_n)\cap R(B)^\perp=\{ 0\}$, for all sufficiently large $n$ (in each case).  
\end{rem}

\begin{rem} 
Among the equivalent conditions  of Theorem \ref{pseudo conv}, we have considered $\sup_n \|B_n\| < \infty$  in a self-contained exposition, using properties of the essential codimension  (Proposition \ref{indexbound}). We point out that the following results in the literature on the convergence of the Moore-Penrose inverse can be used to give alternative proofs.  

\medskip

\noi $i)$ In the infinite dimensional setting Izumino \cite[Lemma 2.2]{I83} proved that for operators $B,B_n \in \cC \cR$ such that $\|B-B_n\|<\|B^\dagger\|^{-1}$ and $\| BB^\dagger - B_nB_n^\dagger \|<1$, then 
$$
\| B_n^\dagger  \| \ \leq \frac{2 \| B^\dagger\| }{1 - \|B-B_n\|\|B^\dagger\|}.
$$ 
This  was adapted by Koliha  \cite[Thm. 1.5]{K01} to the context of $C^*$-algebras.

\medskip

\noi $ii)$ In a work  by Chen, Wei and Sue on the perturbation of the Moore-Penrose inverse in the operator norm, they  proved the following estimate
 (\cite[Thm. 3.2]{CWS96}):
$$
\|B_n^\dagger\| \leq \frac{\|B^\dagger \|}{1 - \frac{1}{2}(3 + \sqrt{5})\|B^\dagger\| \|B_n - B\|},
$$
whenever $R(B_n)\cap R(B)^\perp=\{ 0\}$ and $\|B_n - B\| \leq \frac{3 - \sqrt{5}}{2\| B^\dagger \|}$. 

\end{rem}

%\subsection{Geometric structure of the set of continuous perturbations}\label{sec geom struc pert}
\subsection{Geometric structure of maximal continuity sets}\label{sec geom struc pert}

For $A \in \cC \cR$ we begin by considering the set of perturbations of the form
$$\cC \cR \cap (A + \fS)=\{ B \in \cC \cR : B - A \in \fS\}.$$
This set is endowed with the metric $d_\fS(B_1, B_2)=\|B_1 - B_2 \|_\fS$, for $B_1, B_2 \in \fS$. As we will see below, the essential codimension provides a decomposition of $\cC \cR \cap (A + \fS)$ in connected sets where the Moore-Penrose inverse has  nice continuity properties. It is worth mentioning that the essential codimension was used to give a parametrization of the connected components of infinite-dimensional Grassmannians  or Stiefel manifolds (see \cite{C85} and Remark \ref{on connected comp part isom}).

\begin{rem}\label{charact close range}
The following facts, which are \cite[Thm 3.5 and 3.7]{ECPM23}, will be useful. Given $A, B \in \cC \cR$,
such that $B-A \in \fS$, then
\begin{enumerate}
\item[i)] There exists $G \in \cG \ell_{\FS}$ such that  $B=GA$  if and only if $N(B)=N(A)$. 
\item[ii)] There exist $G,K \in \cG \ell_{\FS}$ such that  $B=GAK^{-1}$  if and only if $[P_{N(B)}:P_{N(A)}]=0$. 
\end{enumerate}
\end{rem}

\begin{nota}\label{nota conv inf}
For a fixed $A \in \cC \cR$, set $n_1=\dim(N(A))$, $n_2=\dim(N(A)^\perp)$, $n_3=\dim(R(A)^\perp)$ and $\mathbb{J}_A=\{ k \in \Z : -\min\{ n_1, \, n_3\} \leq k \leq n_2 \}$. When $n_2=\infty$ we mean $\mathbb{J}_A$ contains all the positive integers; meanwhile when $n_1=n_3=\infty$ we have  $\mathbb{J}_A$ contains all the negative integers.
\end{nota}

\begin{teo}\label{index decomp connected}
Let $A \in \cC \cR$ and let $\fS$ be a symmetrically-normed ideal.  Then $\cC \cR \cap (A + \fS)$ can be decomposed as the following disjoint union
$$
  \cC \cR \cap (A + \fS)=\bigcup_{k \in \mathbb{J}_A}\cC_k(A),
$$
where
\begin{align*}
\cC_k(A) & = \{  B \in \cC \cR : B-A \in \fS , \,  [P_{N(B)}:P_{N(A)}]= k  \} \\
& = \{     B \in \cC \cR  : B- A \in \fS,  \, [P_{R(B)} :  P_{R(A)}]=-k \}.
\end{align*}
The set $\mathbb J_A$ is infinite and $\cC_k(A)\neq \emptyset$, for each $k\in\mathbb J_A$.
Furthermore, the following assertions hold:
\begin{enumerate}
\item[i)] Given $B\in\cC_k(A)$, the action $(G,K) \cdot B=GBK^{-1}$, $G,K \in \cG \ell_\fS$, is well defined and transitive on $\cC_k(A)$.

\item[ii)] $\cC_k(A)$ is  connected. 
\item[iii)] For $\ell,\,k\in \mathbb J_A$ and $B\in \cC_l(A)$ there is a sequence $\{B_n\}_{n\geq 1}$ in $\cC_k(A)$ such that $\|B_n-B\|_\fS\to 0$. In particular, $d_\fS(\cC_k(A) , \cC_l(A)):=\inf \{   \|  B_1-  B_2 \|_\fS  :   B_1 \in \cC_k(A) , \, B_2 \in \cC_l(A) \}=0$, and every $\cC_k(A)$ is dense in $ \cC \cR \cap (A + \fS)$.
\item[iv)] The map $\mu:\cC_k(A) \to \cC \cR$, $\mu(B)=B^\dagger$, is locally Lipschitz.
\item[v)]  Let $\cC_k(A) \subseteq \cC \subseteq \cC \cR \cap (A + \fS)$ endowed with the metric $d_\fS$ be such that the map $\mu: \cC \to \cC \cR$, $\mu(B)=B^\dagger$ is  continuous. Then, 
$\cC=\cC_k(A)$.
\end{enumerate}
\end{teo}
\begin{proof}
In the proof we write $\cC_k:=\cC_k(A)$. For $B \in \cC \cR$, $B-A \in \fS$, notice that $[P_{N(B)}: P_{N(A)}]=0$ if and only if $[P_{R(B)}:P_{R(A)}]=0$. This follows easily from Remark \ref{charact close range} by taking the operator adjoint  and elementary properties of the essential codimension. Next suppose $[P_{N(B)}: P_{N(A)}]=k \neq 0$. If $k>0$, then the  operators defined on $\cH \oplus \C^k$ by $\tilde{A}=A\oplus 0_k$ and $\tilde{B}=B \oplus I_k$ now satisfy $0=[P_{N(\tilde{B})}: P_{N(\tilde{A})}]$. Hence by the previous case, $0=[P_{R(\tilde{B})}: P_{R(\tilde{A})}]=k + [P_{R(B)}: P_{R(A)}]$. The case $k<0$ follows from the property  $[P_{N(B)}: P_{N(A)}]=-  [P_{N(A)}: P_{N(B)}]$. This proves the equivalence between the two conditions defining the sets $\cC_k$. 

In the forthcoming inequalities  we use similar conventions to those of  Notation \ref{nota conv inf} for the cases $n_1=n_3=\infty$, and $n_2=\infty$. According to the definition of the essential codimension,  it follows that   $- \min\{   n_1,   \, n_3\} \leq  [P_{N(B)}: P_{N(A)}]=- [P_{R(B)}: P_{R(A)}] \leq n_2$, for $B \in  \cC \cR \cap (A + \fS)$.  
 Therefore $\cC \cR \cap (A + \fS)$ can be expressed as the disjoint union in the statement. Moreover, note that for $- \min \{  n_1 , \, n_3 \} \leq k \leq n_2$, one can construct 
 operators $B$ such that $A-B\in\fS$ and $[P_{N(B)}: P_{N(A)}]=k$. For instance,  when $n_1 \leq n_3$, there is partial isometry $X_l$ such that $N(X_l)^\perp \subseteq N(A)$, $R(X_l) \subseteq R(A)^\perp$ and $\dim(N(X_l)^\perp)=l$, for $0 \leq l \leq n_1$. Set $B=A + X_l$, which gives $[P_{N(B)}:P_{N(A)}]=-l$. Now for $0<l \leq n_2$, take a subspace $\cS \subseteq N(A)^\perp$ such that $\dim(N(A)^\perp \ominus \cS)=l$. For $B=AP_\cS$,
it holds that $[P_{N(B)}:P_{N(A)}]=l$. The case $n_1 > n_3$ can be treated similarly.  Hence $\cC_k \neq \emptyset$ for $k \in \mathbb{J}_A$. Furthermore, $\mathbb J_A$ is an infinite set 
 because the underlying Hilbert space $\cH$ is infinite dimensional.

\medskip

\noi $i)$ Given $B\in\cC_k$ and $G,K\in\cG\ell_\fS$ then $B_1=(G,K)\cdot B=GBK^{-1}$ is such that $B_1-B\in\fS$, $P_{N(B)}-P_{N(B_1)}\in\fS$ and $[P_{N(B)}:P_{N(B_1)}]=0$ (see Remark \ref{charact close range}). Hence, $A-B_1=A-B+B-B_1\in\fS$ and 
by the properties of the essential codimension (see Section \ref{sec prelis}) we have that 
$[P_{N(B_1)}:P_{N(A)}]=[P_{N(B_1)}:P_{N(B)}]+[P_{N(B)}:P_{N(A)}]=k+0$, so $B_1\in\cC_k$. On the other hand, given $B,\,B_2\in\cC_k$ then $B-B_2=(B-A)+(A-B_2)\in\fS$ and $[P_{N(B)}:P_{N(B_2)}]=[P_{N(B)}:P_{N(A)}]+[P_{N(A)}:P_{N(B_2)}]=k+(-k)=0$. Again, by Remark \ref{charact close range} we get that there exists $G,\,K\in \cG\ell_\fS$ such that $(G,K)\cdot B=B_2$ and the action is transitive on $\cC_k$.

\medskip

\noi $ii)$ From the previous item, every $B \in \cC_k$ is written as $B=GB^{(k)} K^{-1}$, for a fixed $B^{(k)} \in \cC_k$ and $G,K \in \cG \ell_\fS$. According to Remark \ref{restricted groups} $iii)$ there exist $X,Y \in \fS$ such that $e^X=G$ and $e^Y=K$.
Then $\gamma:[0,1] \to \cC_k$ defined by   $\gamma(t)=e^{tX} B^{(k)}  e^{-tY}$ is continuous and $\gamma(0)=B^{(k)}$ and $\gamma(1)=B$.

\medskip

\noi $iii)$ Take $B_1 \in \cC_k$ and  $B_2 \in \cC_l$, and suppose $k > l$. Therefore, $[P_{N(B_2)}: P_{N(B_1)}]=k-l$, so that $\dim(N(B_2) \cap N(B_1)^\perp)\geq k-l$. Let $\cS \subseteq N(B_2)\cap N(B_1)^\perp $ be a subspace of dimension $k-l$, and for $\epsilon >0$ let $B_1^\epsilon= B_2+\frac{\epsilon}{k-l} P_\cS$. One can verify that 
$B_1-B_1^\epsilon\in\fS$ and 
$[P_{N(B_1^\epsilon)}:P_{N(B_1)}]=0$; hence $B_1^\epsilon\in\cC_k$ for $\epsilon>0$ (see Remark \ref{charact close range} and item $i)$ above). Since the singular values satisfy $s_j(B_1^\epsilon - B_2) = \frac{\epsilon}{k-l}$ for $j=1, \ldots, k-l$, and 
$s_j(B_1^\epsilon - B_2)=0$ for $j>k-l$, it follows that $\| B_1 ^\epsilon - B_2 \|_\fS \leq \epsilon$. 

\medskip

\noi $iv)$ Since $\cC_k(A)=\cC_0(B^{(k)})$ we assume, without loss of generality, that $k=0$ and prove that $\mu$ is locally Lipschitz in a neighborhood of $A$. Indeed, take the open ball $\cV:=\{ B \in \cC_0 : \|B- A\|_\fS < \frac{1}{2\|A^\dagger\|} \}$. For $B_1 , B_2 \in \cV$, we consider the same estimate as in  \eqref{bound proof cont}, i.e.
\begin{align}
\|B_2^\dagger - B_1^\dagger \|_\fS & \leq \|B_2\| \|B_2-B_1\|_\fS \| B_1\| +  \|B_2^\dagger \|^2 \|B_2 - B_1 \|_\fS  \|I- B_1B_1^\dagger\|  \nonumber\\
&  + \|I-B_2^\dagger B_2 \| \|B_2  - B_1 \|_\fS \| B_1^\dagger\|^2. \nonumber
\end{align}
Since $B_i \in \cV$, $i=1,2$, we have $\| B_i -A \| \leq \frac{1}{2\|A^\dagger\|}$,  so that $\|B_i\| \leq  \|A\| +  \frac{1}{2\|A^\dagger\|}$. By Proposition \ref{indexbound} it follows that $\|B_i ^\dagger\| \leq \frac{\|A^\dagger\|}{1 - \|A^\dagger\| \|A- B_i\|} \leq 2\| A^\dagger\|$. Therefore, we get
$$
\| B_2 ^\dagger - B_1 ^\dagger \|_\fS \leq   \left[ \left(\|A\| +  \frac{1}{2\|A^\dagger\|}  \right)^2 + 8 \|A^\dagger \|^2  \right] \| B_2 - B_1 \|_\fS .
$$

\medskip

\noi $v)$ If $\cC \neq \cC_k$, then there exists some $l\neq k$ with $B \in \cC_l \cap \cC$. By item  $iii)$ there is a  sequence $\{ B_n\}_{n \geq 1}$ in $\cC_k$ such that $\|B_n - B\|_\fS \to 0$. But this contradicts the continuity of $\mu:\cC \to \cC\cR$ by Theorem \ref{pseudo conv}.
\end{proof}

\begin{rem}
Regarding the Lipschitz condition in the context of infinite-dimensional Hilbert spaces, we recall the following result. Put $\cR_k=\{  B \in \cB(\cH) : \gamma(B)\geq \frac{1}{k} \}$. Then $\| B_1 ^\dagger - B_2 ^\dagger \| \leq 3k^2 \|B_1 - B_2 \|$, for all $B_1 , B_2 \in \cR_k$ (\cite[Lemma 3.10]{CMM09}).
\end{rem}

The following result shows that in case a sequence in $\cC_k(A)$ approaches $A$ in the norm of $\fS$, then it can be modified in a controlled way so that the Moore-Penrose of the modified sequence converges to $A^\dagger$ in the norm of $\fS$.

\begin{cor}
Let $\{B_n\}_{n\geq 1}$ be a sequence in $\cC_k(A)$ for some $0\neq k\in\mathbb J_A$ such that $\|A-B_n\|_\fS\to 0$. 
Then, there exists a sequence $\{C_n\}_{n\geq 1}$ such that $\|C_n\|_\fS\to 0$, $\text{rank}(C_n)=|k|$ for $n\geq 1$, and 
$\|A^\dagger - (B_n+C_n)^\dagger \|_\fS\to 0$.
\end{cor}
\begin{proof}
Assume that $k<0$, so that $\dim(N(A)\cap N(B_n)^\perp)\geq -k$, and let 
$\cS_n\subset N(A)\cap N(B_n)^\perp$ be such that $\dim\cS_n=-k$. If we let $P_n$ denote the orthogonal projection
onto $\cS_n$ then, by construction, 
$\| B_nP_n\|_\fS=\|(A-B_n)P_n\|_\fS\leq \|A-B_n\|_\fS\to 0$.
 Moreover, we also get that $[P_{N(B_n-B_nP_n)}:P_{N(A)}]=0$. Thus, if we let $C_n=-B_nP_n$ then
$\|A-(B_n+C_n)\|_\fS\to 0$ and $[P_{N(B_n+C_n)}:P_{N(A)}]=0$.
By the continuity of the Moore-Penrose inverse in $\cC_0(A)$, we get that $\|A^\dagger - (B_n+C_n)^\dagger \|_\fS\to 0$.

Assume now that $k>0$, so that $\dim(N(B_n)\cap N(A)^\perp)\geq k$, and let 
$\cS_n\subset N(B_n)\cap N(A)^\perp$ be such that $\dim\cS_n=k$. If we let $P_n$ denote the orthogonal projection
onto $\cS_n$ then, by construction, 
$\| AP_n\|_\fS=\|(A-B_n)P_n\|_\fS\leq \|A-B_n\|_\fS\to 0$.
 Moreover, we also get that $[P_{N(B_n+AP_n)}:P_{N(A)}]=0$. Thus, if we let $C_n=AP_n$ then
$\|A-(B_n+C_n)\|_\fS\to 0$ and $[P_{N(B_n+C_n)}:P_{N(A)}]=0$.
Again, by the continuity of the Moore-Penrose inverse in $\cC_0(A)$, we obtain that $\|A^\dagger - (B_n+C_n)^\dagger \|_\fS\to 0$.
\end{proof}

Notice that item $i)$ of Theorem \ref{index decomp connected} says that $\cC_k(A)$ is an orbit by the action of the (product) restricted group $\cG\ell_\fS\times \cG\ell_\fS$, for every $k\in\mathbb J_A$.  In the sequel, we may fix any  $B^{(k)} \in \cC_k(A)$ and write
\begin{align*}
\cC_k (A)& = \{  B \in \cC \cR : B- A \in \fS, \, [P_{N(B)} : P_{N(A)}]=k\} \\
& = \{ GB^{(k)}K^{-1}  : G,K \in \cG \ell_\fS \}. 
\end{align*}
In particular, we can take $B^{(0)}=A$. From item $iv)$, the operators in  the sets $\cC_k(A)$ can be considered as 
 perturbations of the fixed operators $B^{(k)}$. Moreover,  item $v)$ may be interpreted as saying that the set $\cC_0(A)$ is a  maximal subset of $\cC\cR\cap (A+\fS)$ containing $A$ in which the Moore-Penrose inverse is continuous. 

To further study  the structure of $\cC_k(A)$, we now  introduce the maps 
$$
\pi_k : \cG \ell_{\FS} \times \cG \ell_{\FS} \to   \cC_k(A), \, \, \, \, \pi_k(G,K)=GB^{(k)}K^{-1}.
$$
Recall that we consider $\cC_k(A)$ endowed with the topology induced by metric $d_\fS(B,C)=\|B-C \|_\fS$.

\begin{lem}\label{local cross sections}
The map $\pi_k $ admits continuous local cross sections.
\end{lem}
\begin{proof}
We may assume that $k=0$. We show that $\pi_0$ has continuous local cross sections at $A$. The arguments can be adapted to other points $B \in \cC_0(A)$, because $\cC_0(A)$ is an orbit by the continuous action $\pi_0$ of the product group $\cG\ell_\fS\times\cG\ell_\fS$. Thus, we need to find an open neighborhood $\cW$ of $A \in \cC_0(A)$ and a continuous map $\sigma: \cW \to \cG \ell_{\FS} \times \cG \ell_{\FS}$ such that $(\pi_A \circ \sigma) (B)= B$, for all $B \in \cW$. Observe that as an immediate consequence of Theorem \ref{pseudo conv}, there exists an open set $\cW\subset \cC_0(A)$, where the map $\sigma: \cW \to \cB(\cH) \times \cB(\cH)$ defined by
$$
\sigma(B)=(BA^\dagger + (I-P_{R(B)})(I-P_{R(A)})\, ,\,   P_{N(B)} P_{N(A)} +  (I-P_{N(B)})(I-P_{N(A)}) )
$$
is continuous and it takes values on $\cG \ell(\cH) \times \cG \ell(\cH)$. We remark that this section $\sigma$ was first defined in \cite[Prop. 1.1]{ACM05}  in the context of generalized inverses in  $C^*$-algebras (see also \cite[Prop. 5.7]{CMM09}). The fact that this is indeed a section follows analogously in our setting.  
Finally, we may rewrite the first coordinate  as
$
\sigma_1(B)=B(A^\dagger - B^\dagger) + P_{R(A)} (P_{R(B)}- P_{R(A)}) + I;
$
while the second coordinate may be written as $\sigma_2(B)=  P_{N(B)} (P_{N(A)}- P_{N(B)}) + (P_{N(B)}-P_{N(A)})P_{N(A)}  + I$. 
From these expressions together with Lemma \ref{ab y moore penrose}  we find that $\sigma_i(B) \in \cG\ell_{\FS}$, $i=1,2$. 
\end{proof}

To establish the homogeneous space structure of $\cC_k$ induced by the action of $\cG\ell_\fS\times \cG\ell_\fS$ we first consider 
 %some geometrical properties of
 the isotropy subgroup of the action at $A\in\cC_0(A)$. 

\begin{lem}\label{lie subg}
Let $A \in \cC \cR$ and let $\fS$ be a symmetrically-normed ideal.  Then 
$$
\cG_A=\{    (G,K) \in \cG\ell_{\FS} \times  \cG\ell_{\FS} : GA=AK   \}
$$
is a Banach-Lie subgroup of $\cG\ell_{\FS} \times  \cG\ell_{\FS}$. 
\end{lem}
\begin{proof}
We write  $\cG:=\cG\ell_{\FS} \times  \cG\ell_{\FS}$ and $\mathfrak{g}=\fS \times \fS$ for its Lie algebra. Notice that $\cG_A$ is a closed subgroup of $\cG$. Then according to \cite[Prop 8.12]{Up85} we  must show that the closed subalgebra 
\begin{align*}
(T\cG_A)_{(I,I)}=\mathfrak{g}_A & :=\{   (X,Y) \in   \mathfrak{g} : (e^{tX}, e^{tY}) \in \cG_A \text{  for all t} \in \R\} \\
& = \{  (X,Y) \in   \mathfrak{g} :  XA=AY  \} 
\end{align*}
is a closed complemented subspace of $\mathfrak{g}$, and  for every neighborhood $\cT$ of $(0,0)  \in \mathfrak{g}_A$, $\exp_{\cG}(\cT)$ is a neighborhood of $(I,I) \in \cG_A$.

We put $P=P_{R(A)}$ and $Q=P_{N(A)^\perp}$. From $XA=AY$ we see that $XP=AYA^\dagger$, and thus, $P^\perp XP=P^\perp AYA^\dagger=0$. Similarly, $Q YQ^\perp=0$ and $QYQ=A^\dagger XP A=A^\dagger P XP A$. We may represent the elements $ (X,Y) \in \mathfrak{h}$ as follows
$$
X=\begin{pmatrix}      X_{11}   &      X_{12}   \\   0    &   X_{22}    \end{pmatrix}, \, \, \, \, \,  
Y=\begin{pmatrix}     A^\dagger X_{11} A   &      0 \\   Y_{21}    &   Y_{22}    \end{pmatrix}, 
$$
where the first and second matrix representations are with respect to the decompositions $\cH=R(P)\oplus N(P)$  and $\cH=R(Q)\oplus N(Q)$, respectively. Here we identify the operator entries as $X_{11} \in P\FS P$, $X_{12} \in P \FS P^\perp$ and $X_{22} \in P^\perp \FS P^\perp$, and similarly for the operator entries corresponding to $Y$ with respect to $Q$.   From the above representations, it is now clear that $\mathfrak{g}_A$ is closed in $\mathfrak{g}$. Moreover,  a closed supplement of $\mathfrak{g}_A$ in $\mathfrak{g}$ is given by 
$$
\mathfrak{m}=\left\{ (   \begin{pmatrix}    X_{11}   &   0    \\    X_{21} & 0  \end{pmatrix} ,
\begin{pmatrix}       0    &   Y_{12}   \\    0   &    0   \end{pmatrix}
  )  :  X_{11} \in P\FS P, \, X_{21} \in P^\perp \FS P, \, Y_{12} \in Q \FS Q^\perp \right \}.
$$
For the property of the exponential map,  it suffices to show that there exist two open neighborhoods $\cV$ and $\cW$ of $(0,0)  \in \mathfrak{g}$ and $(I,I) \in \cG$, respectively, such that $\exp_\cG: \cV \to \cW$ is bianalytic, and 
 $\exp_\cG(\cV \cap \mathfrak{g}_A)=\cW \cap \cG_A$. The nontrivial inclusion here can be formulated as follows. Given $(G,K) \in \cW \cap \cG_A$, we have to show that $G=e^X$ and $K=e^Y$, for some $(X,Y) \in \cV \cap \mathfrak{g}_A$.   Notice that $GA=AK$, yields $(G-I)^nA=A(K-I)^n$, for all $n \geq 0$. Thus, the logarithm can be defined by the usual series if we take $\cW$ small enough; we further consider $\cV=\exp_\cG^{-1}(\cW)$. Therefore,
  $$\log(G)A=\left(\sum_{n \geq 1}(-1)^{n+1}\frac{(G-I)^n}{n}\right)A=A \left(\sum_{n \geq 1}(-1)^{n+1}\frac{(K-I)^n}{n} \right)= A\log(K)\,.$$  
  Hence we may take $X=\log(G)$ and $Y=\log(K)$ with $X,\,Y\in \cV\cap\mathfrak{g}_A$.
\end{proof}

\begin{teo}\label{estruct variedad y esp homog}\label{teo ck es esp hom}
Let $A \in \cC \cR$ and let $\fS$ be a symmetrically-normed ideal.  Then $\cC_k(A)$ is a real analytic homogeneous space of $\cG \ell_{\FS} \times \cG \ell_{\FS}$. Furthermore, $\cC_k(A)$ is also a real analytic 
submanifold of $A + \fS$ with the same differential structure, whose tangent space at $B\in\cC_k(A)$ is 
$$
(T\cC_k(A))_B=\{  XB - BY :  X, Y \in \FS   \}\subset \fS=(T (A + \fS))_B.
$$ 
\end{teo}
\begin{proof}
\noi We may suppose that $k=0$ and $A=B^{(0)}=B$. Notice that the isotropy subgroup of the action $\pi_0$ is given
by $\cG_A=\{    (G,K) \in \cG\ell_{\FS} \times  \cG\ell_{\FS} : GA=AK   \}$.
From Lemma \ref{lie subg}, $\cG_A$ is a Banach-Lie subgroup of $\cG=\cG \ell_\fS \times \cG \ell_\fS$. Then according to \cite[Thm. 8.19]{Up85} we get that the quotient space $\cG/ \cG_A$ has the structure of real analytic manifold such that $\pi_0 : \cG \to \CA$ is a real analytic submersion.  Furthermore, Lemma \ref{local cross sections} 
implies that $\cC_0(A)$ is homeomorphic to the quotient $\cG/ \cG_A$, so that $\cC_0(A)$ inherits the 
real analytic homogeneous space structure from  $\cG/ \cG_A$.

Now we show that $\CA$ is a submanifold of $A + \fS$.   We first observe that from the previous facts the tangent space $(T\CA)_A$ of $\CA$ at $A$ can be identified as
$$
(T\CA)_A \simeq \mathfrak{g} / \mathfrak{g}_A \simeq \{  XA-AY : X,Y \in \FS  \}\,, 
$$
where we have used that $(T\cG_A)_A=\mathfrak{g}_A=\{(X,Y) \in \mathfrak{g}: XA-AY =0\}$, so that the elements in the quotient space $\mathfrak{g} / \mathfrak{g}_A$ can be identified 
with $\{  XA-AY : X,Y \in \FS  \}$  as above.
We put $P=P_{R(A)}$ and $Q=P_{N(A)^\perp}$. If we take a tangent vector $V=XA-AY$, then $PVQ=P(XA- AY)Q$,   $PV Q^\perp=-P AYQ^\perp$,   $P^\perp VQ=P^\perp XA Q$ and $P^\perp V Q^\perp=0$. Notice that the identity 
$PV Q^\perp=-P AYQ^\perp$ is equivalent to $PV Q^\perp= -A (QYQ^\perp)$,
 meanwhile $P^\perp VQ=P^\perp XA Q$ is equivalent to $P^\perp VQ=(P^\perp X P)AQ$. Thus, the operators $PXAQ=(PXP)AQ$ and $PAYQ=A(QYQ)$, which appear in $PVQ$, are independent of the operators $PV Q^\perp$ and $P^\perp VQ$ (since these last operators can be determined in terms of the independent blocks $QYQ^\perp$ and $P^\perp X P$, respectively).
Therefore, as an operator from  $\cH=R(Q)\oplus N(Q)$   to $\cH=R(P)\oplus N(P)$  we get that $V$ has the form
$$
V= \begin{pmatrix}     Z_{11}   &   Z_{12}  \\   Z_{21}   &   0   \end{pmatrix}.
$$
Conversely,  if $V$ is an operator having such matrix representation, then we may take
$$
X=\begin{pmatrix}  0     &   0  \\   Z_{21}A^\dagger  &  0  \end{pmatrix}, \, \, \, \, 
Y=\begin{pmatrix}  A^\dagger Z_{11}     &   A^\dagger Z_{12}  \\  0  &  0  \end{pmatrix},
$$
which satisfy $V=XA-AY$. Then the tangent space is given by 
\begin{equation}\label{eq tang cc0}
(T\CA)_A =\left\{    \begin{pmatrix}     Z_{11}   &   Z_{12}  \\   Z_{21}   &   0   \end{pmatrix}: Z_{11} \in P \FS Q, \, Z_{12} \in P \FS Q^\perp, \, Z_{21}  \in  P^\perp \FS Q   \right\}.
\end{equation}
From this last representation is now evident that $(T\CA)_A$ is closed in  $\FS$ (tangent space of $A + \FS$), and it has a closed supplement    defined  by
$$
\mathfrak{n}=\left\{    \begin{pmatrix}    0   &  0 \\     0 &   Z_{22}   \end{pmatrix}: Z_{22} \in P^\perp \FS Q^\perp  \right\}.
$$ 
Thus the inclusion map $\iota: \cC_0 \simeq \cG / \cG_A \to  A + \fS$, which is analytic, satisfies that its tangent map has a (closed) complemented range at every point.  This means that $\iota$ is an immersion. Now we recall that the quotient topology on $\CA$ is always stronger than the relative topology, but in this case  both topologies coincide as a consequence of Lemma \ref{local cross sections}. Hence we can apply  \cite[Prop. 8.7]{Up85} to conclude that $\CA$ is a real analytic submanifold of $A+\FS$. Furthermore, the manifold structure as a homogeneous space coincides with that as a submanifold of $A+\FS$. Notice that the tangent map of $\pi_0$ at $(I,I)$ is given by $T_{(I,I)} \pi_0: \fS \times \fS \to (T\CA)_A$, $T_{(I,I)}(X,Y)=XA-AY$. Since this map is a submersion, it follows that $(T\CA)_A=R(T_{(I,I)}\pi_0)=\{ XA- AY : X,Y \in \fS \}$.
\end{proof}

%\subsection{The Moore-Penrose inverse as a real bianalytic map between manifolds}\label{sec real analic pert mp}
\subsection{Real analyticity of the Moore-Penrose inverse}\label{sec real analic pert mp}
%\subsection{Real analytic perturbations}\label{sec real analic pert mp}

In this subsection we show that the Moore-Penrose inverse is a real analytic map between Banach manifolds. 
Hence, we consider the decomposition of $\cC\cR\cap (A+\fS)$ into the connected manifolds $\cC_k(A)$, $k\in\mathbb J_A$, in which the Moore-Penrose inverse is a continuous map. We further consider each maximal continuity set $\cC_k(A)$ endowed with its homogeneous space structure induced by the action of $\cG\ell_\fS\times \cG\ell_\fS$, or equivalently, its submanifold structure.

\begin{rem}\label{analytic comp and real}
Let $E$, $F$ be  complex Banach spaces, let $\Omega \subseteq E$ be an open set  and let $f:\Omega \subseteq E \to F$ be a complex analytic function. Consider $E_0$, $F_0$  real  closed subspaces of $E$ and $F$, respectively, and suppose that $f(E_0)\subseteq F_0$. We claim that the function $f_0: \Omega \cap E_0 \to F_0$, $f_0=f|_{\Omega \cap E_0}$, is real analytic.  
Since $f$ is complex analytic, for every $x_0 \in \Omega$, there exists a convergent power series $\sum_{n \geq 0} f_n$ such that $f(x)=\sum_{n \geq 0} f_n(x-x_0)$ locally at $x_0$,   where each $f_n$ is a continuous $n$-homogeneous polynomial defined on $E$ with values on $F$.  
Denote by $\tilde{f}_n$  the continuous multilinear function associated with $f_n$. 
To prove  our claim, it suffices to check that if $f$ is restricted to $\Omega \cap E_0$,  then the corresponding multilinear functions satisfy  $\tilde{f}_n(E_0^n) \subseteq F_0$.     The case $n=0$, that is $f_0 \in F_0$, follows  by the  assumption $f(E_0)\subseteq F_0$. Next we use that $f$ is complex analytic, so in particular $f$ is a $C^\infty$ function  such that its  derivatives satisfy $f^{(n)}(x_0)=n! \,\tilde{f}_n$ for all $n \geq 1$ (see  \cite[Corol. 1.8]{Up85}).  
 For $x_0, x \in E$, notice that the  Gateux derivative at $x_0$ in the direction of $x$  gives
$$
\tilde{f}_1(x) =  f'(x_0) x=\lim_{t \to 0}  \frac{f(x_0 + t x )   - f(x_0)}{t} \in F_0 ,
$$ 
since $f(E_0) \subset F_0$ and $F_0$ is closed. Hence $\tilde{f}_1 (E_0) \subseteq F_0$. We can now use again the Gateaux derivative of $\tilde{f}_1$ to get $\tilde{f}_2(E_0 \times E_0) \subseteq F_0$. Continuing with this argument we can obtain  $\tilde{f}_n(E_0^n) \subseteq F_0$, for all $n\geq 1$.
\end{rem}

\begin{lem}\label{smooth unitaries}
Let $\cS \subseteq \cH$ be a closed subspace. Then there exists 
a real analytic map $\cG \ell_{\FS} \to \cU_{\FS}$, $T \mapsto U_T$, such that 
$P_{T(\cS)}=U_T P_\cS U_T ^*$, for $T\in \cG \ell_{\FS}$. 
\end{lem}

\begin{proof}
Let $T\in\cG\ell_\fS$ and set $P=P_{T(\cS)}$, $Q=T P_{\cS}  T^{-1}$. Then, 
$$
T_0=PQ + (I-P)(I-Q)=Q+(I-P)(I-Q)
$$ 
satisfies $PT_0=T_0Q=Q$, $T_0|_{R(Q)}=I|_{R(Q)}$ and $T_0|_{N(Q)}=(I-P)|_{N(Q)}$ is an isomorphism between $N(Q)$ and $N(P)$. Since $Q$ is an oblique projection we get that $\cH$ is the (non-orthogonal) direct sum of $R(Q)$ and $N(Q)$; thus, $T_0$ is an invertible operator. 

We now show  $T_0 \in \cG \ell_\fS$.  
Since $T\in\cG\ell_\fS$ then, by \cite[Thm. 3.5]{ECPM23} we get that $P-P_\cS=P_{R(TP_\cS)}-P_{R(P_\cS)}\in \fS$ (and furthermore, $[P:P_\cS]=0$). On the other hand, $Q-P_\cS=TP_\cS(T^{-1}-I)+(T-I)P_\cS\in\fS$.
 These facts imply that $Q-P\in\fS$ and hence $T_0 - I=P(Q-P)\in \fS$.

 Now set $T_1=T_0T$, which satisfies $T_1 \in \cG \ell_{\FS}$ and $T_1P_\cS T_1^{-1}=P$. Then, $T_1P_\cS = PT_1$ and $P_\cS T_1^*=T_1^* P$ gives $T_1^*T_1 P_\cS=P_\cS T_1^* T_1$. Thus, we get $|T_1|P_\cS=P_\cS|T_1|$. The unitary $U_T=T_1 |T_1|^{-1}$  now can be seen to satisfy $U_T \in \cU_{\FS}$ and $U_T P_\cS U_T^*=P_{T(\cS)}$.

To obtain that the map $T \mapsto U_T$ is real analytic and thereby complete the proof, we
express this map as a composition of real analytic maps. Recall that $\tilde{\fS}=\{  X + \lambda I : X \in \fS, \, \lambda \in \C \}$ is the unitalization of $\fS$, and consider its self-adjoint part, i.e. $\tilde{\fS}_{sa}:=\{ X + \lambda I : X=X^*, \, \lambda \in \R \}$ (see Remark \ref{restricted groups}).

Since $T$ is invertible, $P_\cS T^*TP_\cS \geq c P_\cS$, $c=\min_{\lambda \in \sigma(|T|)} \lambda^2$, so that $P_\cS T^*T|_\cS$ is invertible on $\cS$ and $P_\cS (P_\cS T^*TP_\cS)^\dagger P_\cS=(P_\cS T^*T|_\cS)^{-1}P_\cS$.
Then, the map $\cG \ell_\fS  \to   \tilde{\fS}_{sa}$, $T \mapsto P_{T(\cS)}=TP_\cS(P_\cS T^*T|_\cS)^{-1}P_\cS T^*$ is clearly real analytic since the operator adjoint, multiplication and inversion maps are real analytic. If $Q_T=TP_\cS T^{-1}$, then the map 
$\cG \ell_\fS \times \tilde{\fS}_{sa} \to \tilde{\fS}$, $(T,R)\mapsto Q_T + (I-R)(I - Q_T)$ is also real analytic. Therefore,
$$
T_0 : \cG \ell_\fS  \to \tilde{\fS}, \, \, \, \, T_0(T)=Q_T  +(I-P_{T(\cS)})(I- Q_T), 
$$ 
is real analytic. Since $\cG \ell_\fS$ is a submanifold of $\tilde{\fS}$ as we state in Remark \ref{restricted groups} $ii)$,  and $T_0(T) \in \cG \ell_\fS$ for every $T  \in \cG \ell_\fS$, it follows that  $T_0 :  \cG \ell_\fS \to  \cG \ell_\fS$ turns out to be real analytic. We conclude that $T_1 :  \cG \ell_\fS \to  \cG \ell_\fS$, $T_1(T)= T_0(T) T$, is real analytic. 

On the other hand, notice that $\sigma(Z)=\sigma_{\tilde{\fS}} (Z)$  for any $Z=(Z-I) + I \in\cG\ell_\fS\subset \tilde{\fS}$, where the left-hand spectrum is the spectrum as an operator on $\cH$ and the right-hand spectrum is the spectrum in the Banach algebra $\tilde{\fS}$. Put $\C_+ := \{ z \in \C : \Re(z) >0 \}$, and then take the open subset $\Omega=\{ Z \in \tilde{\fS} : \sigma(Z) \subseteq  \C_+  \}$ of $\tilde{\fS}$. Now we apply Remark \ref{analytic comp and real} with $E=F=\tilde{\fS}$, $E_0=F_0=\tilde{\fS}_{sa}$ and the map $f: \Omega \to \tilde{\fS}$, $f(Z)=Z^{1/2}$, which is certainly complex analytic by the holomorphic functional calculus in the Banach algebra $\tilde{\fS}$ (see, e.g., \cite[p. 30]{Up85}). 
We thus get that $f_0: \Omega \cap \tilde{\fS}_{sa} \to \tilde{\fS}_{sa}$, $f_0(Z)=Z^{1/2}$, is real analytic. For $T \in \cG\ell_\fS$ notice that $|T| -I=(|T| + I)^{-1}(|T|^{2} - I) \in \fS$, whence $T^*T=|T|^2,\,|T| \in \cG\ell_\fS $. 
Hence the maps $\cG \ell_\fS \to  \cG \ell_\fS$, $T\mapsto |T|=f_0(T^*T)$ and $T \mapsto T |T|^{-1}$ are  real analytic. Since $\cU_\fS$ is a submanifold $\cG \ell_\fS$, we can co-restrict the previous map and find that 
$\cG \ell_\fS \to  \cU_\fS$, $T \mapsto T |T|^{-1}$, is real analytic. Finally, we obtain that the map $\cG \ell_\fS \to \cU_\fS$, $T \mapsto U_T$, where $U_T=T_1(T) |T_1(T)|^{-1}$ is real analytic.
 \end{proof}

\begin{rem}\label{dagger inv}
For fixed  $A \in \cC \cR$ we may also express $ \cC \cR  \cap (A^\dagger + \fS)$ as follows
$$
  \cC \cR \cap (A^\dagger  + \fS)=\bigcup_{k \in \mathbb{J}_A}\cC_k (A^\dagger),
$$
where $\mathbb{J}_A=\mathbb{J}_{A^\dagger}$ is the set defined in Notation \ref{nota conv inf}, and we set
$$
\cC_k (A^\dagger) := \{  B \in \cC \cR : B - A^\dagger \in \fS , \,  [P_{N(B)}:P_{N(A^\dagger)}]= k  \}.
$$
In particular, by Theorem \ref{estruct variedad y esp homog}, $\cC_k(A^\dagger)$ is also a real 
analytic homogeneous space of $\cG=\cG\ell_\fS\times\cG\ell_\fS$ and a  submanifold of $A^\dagger +\fS$. From Lemma \ref{ab y moore penrose}, we know that $B - A \in \fS$ if and only if $B^\dagger - A^\dagger \in \fS$. Thus we can consider the bijection $\mu:   \cC \cR \cap (A  + \fS)  \to   \cC \cR \cap (A^\dagger  + \fS)$, $\mu(B)=B^\dagger$. Since $[P_{N(B^\dagger)} :  P_{N(A^\dagger)}]=-[P_{N(B^\dagger)^\perp} :  P_{N(A^\dagger)^\perp}]=-[P_{R(B)} :  P_{R(A)}]=[P_{N(B)} :  P_{N(A)}]$  (the last equality follows from Theorem \ref{index decomp connected}), we get  $\mu(\cC_k(A))=\cC_k (A^\dagger)$, for $k \in \mathbb{J}_A$. Hence $\mu: \cC_k(A) \to \cC_k (A^\dagger)$, $\mu(B)=B^\dagger$, is a homeomorphism.
\end{rem}

 The previous remark concerns the map defined by the Moore-Penrose inverse at a topological level. The geometric structures we have constructed in Theorem \ref{teo ck es esp hom} allow us to consider the Moore-Penrose inverse as a map between Banach manifolds. We can now give one of our main results about the Moore-Penrose inverse in this framework.

\begin{teo}\label{teo pseudo in real analytic}
Let $A \in \cC \cR$ and let $\fS$ be a symmetrically-normed ideal. Then for every $k\in\mathbb{J}_A$ the map 
$$
\mu: \cC_k(A) \to \cC_k(A^\dagger),  \, \, \, \mu(B)=B^\dagger,
$$
 is real bianalytic between these Banach manifolds, and its tangent map is given by
$$
(T _B \mu)(V)= - B^\dagger V B^\dagger    +   (B^*B)^\dagger V^* (I-BB^\dagger)+  (I- B^\dagger B)V^*(BB^*)^\dagger ,
$$   
for  $B \in \cC_k(A)$ and  $V \in (T \cC_k(A))_B$.
\end{teo}
\begin{proof}
We may assume that $k=0$. As we observed in Remark \ref{dagger inv} the map $\mu$ in the statement is a homeomorphism. Now we recall two properties of the Moore-Penrose inverse. For $C \in \cC \cR$, and $U,V$ unitary operators, it holds that $(UCV)^\dagger=V^*C^\dagger U^*$, and that $C^\dagger$ is determined  by the conditions $C^\dagger|_{R(C)^\perp}=0$ and $C^\dagger|_{R(C)}=(C|_{R(C)})^{-1}$, where the inverse here is given by  $(C|_{R(C)})^{-1}:R(C) \to N(C)^\perp$ satisfying $(C|_{R(C)})^{-1} C|_{N(C)^\perp}=I_{N(C)^\perp}$ and $C (C|_{R(C)})^{-1}=I_{R(C)}$. According to Lemma \ref{smooth unitaries}, for every $(G,K)\in\cG\ell_\fS\times \cG\ell_\fS$ there are unitary operators $U_K, U_G \in \cU_{\FS}$ such that  $U_K(N(A)^\perp)=K(N(A)^\perp)$, and $U_G(R(A))=G(R(A))$. Then, 
\begin{align*}
(GAK^{-1})^\dagger & =U_K (U_G^*(GAK^{-1})U_K)^\dagger U_G^* \\
& =U_K \begin{pmatrix}     A^\dagger   &    0   \\  0 &   0    \end{pmatrix} 
\begin{pmatrix} [U_G^*(GAK^{-1})U_K A^\dagger ]^{-1}   &   0 \\ 0   &   0  \end{pmatrix}U_G^*\,,
\end{align*} where the block matrix representation of factor to the right is with respect to the decomposition $\cH=R(A)\oplus R(A)^\perp$ and the block matrix representation of the factor to the left is as an operator acting from $\cH=R(A)\oplus R(A)^\perp$ into $\cH=N(A)^\perp\oplus N(A)$.
The operator obtained by restriction $U_G^*(GAK^{-1})U_K A^\dagger : R(A) \to R(A)$ is invertible, and using that $P_{R(A)}=AA^\dagger$ we further get that $U_G^*(GAK^{-1})U_K A^\dagger  - P_{R(A)} \in \fS$. We write $\fS(R(A))$ for the corresponding symmetrically-normed ideal $\fS$ on $R(A)$.  By considering also the restricted invertible group on $R(A)$, 
 i.e. $\cG \ell_\fS (R(A)):=\{ G \in \cG\ell(R(A)) : G - I_{R(A)} \in \fS(R(A)) \}$, we can use that the inversion map is real analytic on this Lie group.
Since $\cU_{\FS}$  is a Lie group, the inversion is real analytic, and using  that $K \mapsto U_K$ and $G \mapsto U_G$ are real analytic by  Lemma \ref{smooth unitaries}, we get from the above expression that the map $\tilde{\pi}_0: \cG \ell_\fS \times \cG \ell_\fS \to \cC_0(A ^\dagger) $, $\tilde{\pi}_0(G,K) = (GAK^{-1})^\dagger$ is real analytic. By Theorem \ref{teo ck es esp hom} we conclude that $\pi_0:\cG\ell_\fS\times \cG\ell_\fS\rightarrow \cC_0(A)$ is real analytic, and hence it has real analytic local cross sections (\cite[Corol. 8.3]{Up85}). Thus for every $B \in \cC_0(A)$, there is an open set $B\in \cW \subseteq \cC_0(A)$, and a real analytic map $s:\cW \to \cG\ell_\fS\times\cG\ell_\fS$ such that  $\pi_0  \circ s=id|_{\cW}$. Therefore, 
we can write locally $\mu= \tilde{\pi}_0 \circ s$. This shows that $\mu$ is real analytic. Clearly, $\mu$ becomes a bianalytic map between these manifolds. 

Now we  can compute the tangent map at $B\in\cC_0(A)$ in the direction of a vector  $V=XB - BY\in (T \cC_0(A))_B$. Take the curve $\gamma(t)=e^{tX} B e^{-tY} \in \CA$, for some $X,Y \in \FS$,  and $t \in (-1, 1)$, which satisfies 
 $\gamma(0)=B$, $\dot{\gamma}(0)=V$. By Wedin's formula \eqref{stewart identity} and the continuity of $\mu$ in $\cC_0(A)$ we get that
\begin{align*}
(T_B \mu )(V)& =\frac{d}{dt}\Big\rvert_{t=0} \mu(\gamma(t)) \\
& =\lim_{t \to 0} \frac{1}{t} \{ -\gamma(t)^\dagger(\gamma(t)-\gamma(0))\gamma(0)^\dagger + (\gamma(t)^* \gamma(t))^\dagger (\gamma(t)^* - \gamma(0)^*)(I - \gamma(0)\gamma(0)^\dagger) + \\ & + (I-\gamma(t)^\dagger \gamma(t))(\gamma(t)^* - \gamma(0)^*) (\gamma(0)\gamma(0)^*)^\dagger  \} \\
& =    - B^\dagger V B^\dagger    +   (B^*B)^\dagger V^* (I-BB^\dagger)+  (I- B^\dagger B)V^*(BB^*)^\dagger .
\end{align*}

\end{proof}

\begin{rem}
\noi $i)$  For $A \in \cC \cR$, $G,K \in \cG \ell_{\FS}$, we may  compute $(GAK^{-1})^{\dagger}$ using the following formula proved in \cite[Thm. 2]{GP91}:
$$
(GAK^{-1})^\dagger=(AK^{-1})^*[AK^{-1}(AK^{-1})^* + I- AA^\dagger ]^{-1} A 
[(GA)^*GA + I- A^\dagger A]^{-1} (GA)^*.
$$
A straightforward consequence is that  $(G,K) \mapsto (GAK^{-1})^\dagger$ is a real analytic map. Thus, this gives another proof of the fact that $\mu$ is a real analytic map. However, the approach considered in the proof of Theorem \ref{teo pseudo in real analytic}, 
which is based on Lemma \ref{smooth unitaries}, is also needed in Section \ref{polar decomposition}.

\medskip

\noi $ii)$ Since $\cG \ell_\fS$ is also a complex Lie group, the same proof of Theorem \ref{estruct variedad y esp homog} also shows that $\cC_k$, $k \in \mathbb{J}_A$, are complex analytic 
homogeneous spaces and submanifolds. The operator adjoint that shows up in the proofs of Lemma \ref{smooth unitaries} and Theorem \ref{teo pseudo in real analytic}  implies that $\mu$ cannot be complex analytic.  We refer to \cite{LR} for a study of other generalized inverses as complex analytic mappings operator valued mappings of a complex variable.
\end{rem}

\section{Polar decomposition}\label{polar decomposition}

We first study the real analyticity of operator monotone functions defined on invertible positive perturbations by symmetrically-normed ideals. Then we develop a geometric study of the closed range positive operator perturbations by an ideal. This leads to a decomposition into connected sets, which are real analytic homogeneous  spaces (congruence orbits) and submanifolds. In the next section we use these results, and also the ones from the previous sections, to show that the maps given by the operator modulus and the partial isometry in the polar decomposition are real analytic fiber bundles.

%\subsection{Perturbations and operator monotone functions}
\subsection{Real analyticity of operator monotone functional calculus}

We revisit the Ando-van Hemmen theory on the perturbation problem for operator monotone functions \cite{AvH}. 
We are interested in a rather different problem, related to the real analyticity of the operator monotone functional calculus with respect to operator ideal perturbations. Nevertheless, our approach is deeply influenced by the techniques from \cite{AvH}. Recall that a  function $f:[0,\infty)\rightarrow \R$ is said to be operator monotone if $C,D\in \cB(\cH)^+$ are such that $C \leq D$, then we have that $f(C)\leq f(D)$. In this case $f$ belongs to the Pick class. Hence, there exist a positive Borel measure $\nu$ on $(0,\infty)$ such that $\int_0^\infty (t^2+1)^{-1}\, d\nu(t)$ is finite, $\alpha\in\R$ and $\beta\in [0,\infty)$ such that 
\begin{equation}\label{eq int rep omf}
f(\la)=\alpha+\beta\,\lambda-\int_0^\infty \left(\frac{1}{t+\la}-\frac{t}{t^2+1}\right)\ d\nu(t)\quad \text{for} \quad \la\in [0,\infty)\,.\end{equation}
Moreover, notice that the expression for $f(\la)$ above is well defined for $\la\in \C\setminus (-\infty,0)$.
For example, if $f(\la)=\la^{1/2}$, for $\la\in [0,\infty)$, then 
\begin{equation}\label{eq rep raiz cuad}
\la^{1/2}= \frac{1}{\sqrt 2} -\int_0^\infty \left(\frac{1}{t+\la}-\frac{t}{t^2+1}\right) \frac{t^{1/2}}{\pi} \,dt\quad \text{for} \quad \la\in [0,\infty)\,.
\end{equation}
Using the integral representation in Eq. \eqref{eq int rep omf} for the operator monotone function $f(\la)$, then given $C \geq 0$ we can represent the self-adjoint operator
$$
f(C)=\alpha\,I+\beta\,C-\int_0^\infty \left((t\,I+C)^{-1}-\frac{t}{t^2+1}\,I\right)\ d\nu(t)\,.
$$

In what follows we use the notation $\cG \ell ^+:=\cG \ell^+ (\cH)=\{ G \in \cG \ell(\cH) : G > 0 \}$. For a fixed $C \in \cG \ell^+$  notice that $\cG \ell^+ \cap (C + \fS)$ is an open subset of the affine space $C + \fS$, where the topology is,  as usual, defined by the metric $d_\fS(B_1 , B_2)=\| B_1 - B_2 \|_\fS$, $B_1 , B_2 \in C + \fS$.  
We recall that $\fS_{sa}$ denotes the set of self-adjoint operators in $\fS$.

\begin{rem}\label{rem AvH1}
Consider a monotone operator function represented as in Eq. \eqref{eq int rep omf} and let $C,\,D\in\cG\ell^+$ be such that $D-C\in\fS$. In \cite{AvH} Ando and van-Hemmen showed, in particular, that $f(D)-f(C)$ lies in the maximal ideal associated to the symmetric norming function induced by $\fS$; for the notion of maximal ideal see \cite{GK60}. In this context, they noticed that
\begin{eqnarray*}
f(D)-f(C)&=&\beta(D-C)+\int_0^\infty \left((t\,I+C)^{-1}-(t\,I+D)^{-1}\right)\ d\nu(t) \\
   &=&\beta(D-C)+\int_0^\infty (t\,I+C)^{-1}( D-C)(t\,I+D)^{-1}\ d\nu(t) 
\end{eqnarray*}
The previous identities suggest to consider the operator valued function $h:[0,\infty)\rightarrow \fS_{sa}$ given by
$$
h(t)=(t\,I+C)^{-1}( D-C)(t\,I+D)^{-1}\quad \text{for} \quad t\in [0,\infty)\,.
$$
In this case it is straightforward to check that: 
\begin{enumerate}
%\item \label{item1}$h:[0,\infty)\rightarrow \fS_{sa}$ is a continuous function ($\fS_{sa}$ endowed with %$d_{\fS})$) such that $\lim_{t\rightarrow \infty}h(t)=0$;
\item \label{item2} $\|h(t)\|_\fS\leq q(t)$, where $q(t)$ is a bounded, continuous, positive and non-increasing function such that  $\int_{0}^{\infty} q(t)\ d\nu(t)<\infty$ (e.g. $q(t)=\|D-C\|_{\fS} ((t+\gamma_C)\,(t+\gamma_D))^{-1}$, $t\geq 0$);
\item \label{item3} If $\delta>0$, then $\| h(t+\delta)-h(t)\|_{\fS}\leq \alpha \ \delta \ r(t)$,  where $\alpha>0$ and $r:[0,\infty)\rightarrow [0,\infty)$ is a  continuous and non-increasing function, such that $$\lim_{t\rightarrow \infty}r(t)=0 \ \ \text{and} \ \ \int_0^\infty r(t)\ d\nu(t)<\infty\,.$$
\end{enumerate}
We point out that in the case above we can take $r(t)$ to be
$$r(t)= ((t+\gamma_C)(t+\gamma_D))^{-1}\,\|D-C\|_\fS \, ((t+\gamma_C)^{-1}+(t+\gamma_D)^{-1})\ , \ \ t\in [0,\infty)\,.$$
\end{rem}
The following lemma is an integral formulation of the fact that absolute convergence implies convergence in 
symmetrically-normed ideals.

\begin{lem}\label{lem conv abs}
Let $\fS$ be symmetrically-normed ideal, let $\nu$ be a positive Borel measure on $(0,\infty)$ and let $h:[0,\infty)\rightarrow \fS_{sa}$ be a function satisfying items 1 and 2 in Remark \ref{rem AvH1}. Then, 
$$
\int_0^\infty h(t)\ d\nu(t)\in \fS_{sa} \quad \text{and}\quad \left\|\int_0^\infty h(t)\ d\nu(t)\right\|_{\fS}\leq 
\int_0^\infty \|h(t)\|_{\fS}\ d\nu(t)\,.
$$
\end{lem}
\begin{proof}
See Section \ref{apendixity}.
\end{proof}

\begin{pro}\label{pro AvH1}
Let $C,D\in \cG\ell^+$ be such that $D-C\in\fS$, and let $f:[0,\infty)\rightarrow \R$ be an operator monotone function with integral representation as in Eq. \eqref{eq int rep omf}. Then, we have that $f(D)-f(C)\in\fS_{sa}$ is such that $$\|f(D)-f(C)\|_\fS\leq \|D-C\|_\fS\ \left ( \beta +\int_0^\infty \frac{1}{(t+\gamma_C)\,(t+\gamma_D)}\ d\nu(t)\right)\,.$$
\end{pro}

\begin{proof}
With the notation above, and arguing as in Remark \ref{rem AvH1} we get that
\begin{equation}\label{eq f diffcd1}
f(D)-f(C)=\beta\,(D-C)+\int_0^\infty h(t) \ d\nu(t) \quad \text{for} \quad h(t)=
(t\,I+C)^{-1}( D-C)(t\,I+D)^{-1}\,.
\end{equation}
Notice that $h(t)$ fulfills the conditions in items \ref{item2} and \ref{item3} in Remark \ref{rem AvH1} with 
$\|h(t)\|\leq \|D-C\|_{\fS} ((t+\gamma_C)\,(t+\gamma_D))^{-1}$, for $t\geq 0$. Then, by Lemma \ref{lem conv abs} we see that 
$$
\int_0^\infty h(t)\ d\nu(t)\in \fS_{sa} \quad \text{and}\quad \left\|\int_0^\infty h(t)\ d\nu(t)\right\|_{\fS}
\leq 
\int_0^\infty \frac{\|D-C\|_\fS}{(t+\gamma_C)\,(t+\gamma_D)}\ d\nu(t)\,.
$$
The result now follows from Eq. \eqref{eq f diffcd1} and the last remarks.
\end{proof}

%\begin{teo}[\cite{AvH}] \label{teoAvH1} Let $C,\,D$ be positive operators and assume that $C+D\geq \eta\,I$, for some $\eta>0$.
%Then, for any operator monotone function $f:[0,\infty)\rightarrow \R$ and any symetrically normed ideal $\fS$ with norming function $\Phi$ we have that 
%$$
%\| f(C)-f(D)\|_\Phi\leq 2\, \left( \frac{f(\eta/2)-f(0)}{\eta}\right)\ \|C-D\|_\Phi\,.
%$$
%\end{teo}
%
%%\begin{rem}\label{rem teo AvH1}
%Notice that the hypothesis that $C+D\geq \eta\,I$, for some $\eta>0$ in Theorem \ref{teoAvH1} cannot be avoided.
%We recall the following example from \cite{AvH}.  Let $\fS_2$ denote the ideal of Hilbert-Schmidt operators acting on $\cH$, so that $\Phi=\|\ \|_2$ is the (euclidean) $2$-norm. Consider the operator monotone function $f(\la)=\la^{1/2}$, for $\la\in [0,\infty)$. Let $C\in\fS_2$ be a positive operator, that is not a trace class operator. If we let $D=0$ then there does not exist $\eta>0$ such that $C=C+D\geq \eta\,I$. On the other hand, $C-D\in \fS_2$ but $C^{1/2}-D^{1/2}=C^{1/2}\notin \fS_2$. In particular, the expression $\|f(C)-f(D)\|_\Phi$ is not defined even when $C-D\in\fS_2$.
%%\end{rem}

Next we study the continuity of the functional calculus induced by a (fixed) monotone operator function, with respect to certain perturbations of a positive closed range operator. To do this, we consider the following facts.

%\begin{rem}\label{rem hay cont AvH}
%As a consequence of Proposition \ref{pro AvH1} given $C\in\cG\ell^+$ and $\eta>0$ we conclude that the operator function $D\mapsto f(D)$ is continuous from $\{D\in \cG\ell^+:\ C-D\in\fS , \, \gamma_D\geq \eta\}$ to $f(C)+\fS_{sa}$, when these sets are endowed with the distance function $d_\fS(D_1,D_2)=\|D_1-D_2\|_\fS$. Below we extend this fact to some other domains. 
%\end{rem}

\begin{rem}\label{section grass}
Let $P$, $Q_n$ be orthogonal projections such that $P-Q_n\in \fS$, for all $n\geq 1$, and $\|P-Q_n\|_\fS\rightarrow 0$. Then, there exists a sequence $\{ U_n \}_{n \geq 1}$ in $\cU_\fS$ such that $U_nQ_nU_n^*=P$, for $n\geq 1$, and $\|U_n-I\|_\fS\rightarrow 0$. Indeed, this is a consequence of 
\cite[Proposition 2.2.]{AL08} for the ideal $\fS_2$ of Hilbert-Schmidt operators; there it is shown the existence of continuous local cross sections for the map $\cU_{\fS_2}\ni U\mapsto U^*PU\subset P+(\fS_2)_{sa}$, where $\cU_{\fS_2}$ and $P+(\fS_2)_{sa}$ are endowed with the Hilbert-Schmidt  metric $d_{2}(C,D)=\|C-D\|_{2}$. The general case of a symmetrically-normed operator ideal follows with a straightforward adaption of the proof of the previous result.
\end{rem}

In what follows we let $\cC\cR^+=\{C\in\cC\cR:\ C\geq 0\}$. 

\begin{cor}\label{hay cont en p0}
Let $f:[0,\infty)\rightarrow \R$ be an operator monotone function and let 
$\fS$ be a symmetrically-normed ideal. Fix $C\in\cC\cR^+$ and let $\{D_n\}_{n\geq 1}$ be a sequence in $\cC\cR^+$ such that $C-D_n\in\fS$, $[P_{N(C)}:P_{N(D_n)}]=0$ for $n\geq 1$, and $\|D_n-C\|_\fS\rightarrow 0$. Then, there exists $n_0\geq 1$ such that $f(D_n)-f(C)\in\fS$, for $n\geq n_0$, and $\|f(D_n)-f(C)\|_\fS\rightarrow 0$.
\end{cor}
\begin{proof} Let $\{D_n\}_{n\geq 1}$ be as above and assume futher that $N(D_n)=N(C)$ for $n\geq 1$. 
Then 
the restrictions $D_n|_{\cH_0},\,C|_{\cH_0}\in\cB(\cH_0)$ are positive invertible operators acting on $\cH_0=R(C)$ such that $\|D_n|_{\cH_0} - C|_{\cH_0}\|_\fS\rightarrow 0$.
Let $n_0\geq 1$ be such that $\|D_n-C\|_\fS\leq \gamma_{C|_{\cH_0}}/2$, 
so that $\gamma_{D_n|_{\cH_0}}>\gamma_{C|_{\cH_0}}/2$, for $n\geq n_0$. In this case $$((t+\gamma_{D_n|_{\cH_0}})\,(t+\gamma_{C|_{\cH_0}}))^{-1}\leq 
4\, (t+\gamma_{C|_{\cH_0}})^{-2}\ , \ t\geq 0\,.$$
%By Theorem \ref{teoAvH1} (see Remark \ref{rem hay cont AvH} and notice that $C|_{\cH_0}\geq \eta\,I_{\cH_0}$ for some $\eta>0$) 
By Proposition \ref{pro AvH1} we get that $f(D_n|_{\cH_0})-f(C|_{\cH_0})\in\fS(\cH_0)$ for $n\geq n_0$, and that 
$\|f(D_n|_{\cH_0})-f(C|_{\cH_0})\|_\fS\rightarrow 0$. Now notice that 
$f(C)=f(C|_{\cH_0})+f(0)\,(I-P_{\cH_0})$ and similarly for $f(D_n)$, $n\geq 1$; thus, 
$f(D_n)-f(C)\in\fS$ for $n\geq n_0$ and 
$\|f(D_n)-f(C)\|_{\fS}=\|f(D_n|_{\cH_0})-f(C|_{\cH_0})\|_\fS\rightarrow 0$.

Consider now a general sequence $\{D_n\}_{n\geq 1}$ as in the statement above. 
Notice that by hypothesis and Theorem \ref{pseudo conv}, we get that $\|D_n^\dagger-C^\dagger \|_\fS\rightarrow 0$ and hence 
$\|P_{N(D_n)}-P_{N(C)}\|_\fS=\| D_n^\dagger D_n -C^\dagger C \|_\fS\rightarrow 0$. Let $\{U_n\}_{n\geq 1}$ be a sequence in $\cG\ell_\fS$ such that $U_n P_{N(D_n)} U_n^*=P_{N(C)}$ and $\|U_n-I\|_\fS\rightarrow 0$ (see Remark \ref{section grass}). Hence, $U_nD_nU_n^*-C\in\fS$, $N(U_nD_nU_n^*)=N(C)$ and $\|U_nD_nU_n^*-C\|_\fS\rightarrow 0$. By the first part of the proof we conclude that 
$\|f(U_nD_nU_n^*)-f(C)\|_\fS\rightarrow 0$. On the other hand, $f(U_nD_nU_n^*)=U_nf(D_n)U_n^*$ so $\|f(U_nD_nU_n^*)- f(D_n)\|_\fS\rightarrow 0$. The previous facts imply that 
$$
\| f(D_n)-f(C)\|_\fS\leq \|f(D_n)-f(U_nD_nU_n^*)\|_\fS+\| f(U_nD_nU_n^*)- f(C)\|_\fS\rightarrow 0\,.
 \qedhere$$
\end{proof}

%\begin{rem}
In the proof of Lemma \ref{smooth unitaries} we have used that the square root is a real analytic map in the set 
%$\{  G \in \cG \ell_\fS :   G >0  \}$
$\cG \ell_\fS \cap \cG\ell^+$, by considering the enveloping unital Banach algebra $\tilde{\fS}$
%=\{  X + \lambda I : X \in \fS, \, \lambda \in \C \}$ endowed with the norm $\|X + \lambda I\|=\|X\|_\Phi+|\lambda|$  
(see Remark \ref{restricted groups}). Thus, the argument in that proof cannot be repeated when
the domain is changed to the set of perturbations $\{C+K \in \cG\ell^+: K \in \fS_{sa}\}$ endowed with the distance $d_{\fS}$, where $C\in\cG\ell^+$ is some fixed positive invertible operator.  Below we prove that the  square root, and moreover any operator monotone function, is real analytic on these more general domains. 
%\end{rem}

\begin{teo}\label{teo sobre AvH}
Let $f:[0,\infty)\rightarrow \R$ be an operator monotone function on $[0,\infty)$ with integral representation as in Eq. \eqref{eq int rep omf} and let $C\in\cG\ell^+$.
Consider $f:(C+\fS)\cap \cG\ell^+\rightarrow f(C)+\fS_{sa}$ given by $(C+\fS)\cap \cG\ell^+\ni D\mapsto f(D)\in f(C)+\fS_{sa}$. Then, $f$ is a real analytic function. Moreover, for $\|D-C\|_\fS<  \gamma_C$ we get a local series representation 
$f(D)=f(C)+\sum_{n=1}^\infty f_n(D-C)$, where
$$
f_1(D-C)=\beta\,(D-C)+\int_0^\infty  (tI+C)^{-1}\,(D-C)\,(t\,I+C)^{-1}\ d\nu(t)\in \fS_{sa} \, ,
$$
$$f_n(D-C)=(-1)^{n+1}\, \int_0^\infty  \left((tI+C)^{-1}(D-C)\right)^{n} (t\,I+C)^{-1}\ d\nu(t)\in \fS_{sa} \, , \ n\geq 2\,.$$
\end{teo}

\begin{proof}
Given $C$ as above we show that $f$ admits a local power series around $C$, as in 
Remark \ref{analytic comp and real}. Indeed, assume that $D\in \cG\ell^+$ is such that $D-C\in \fS$. By Proposition \ref{pro AvH1} we get that 
$f(D)\in f(C)+\fS_{sa}$. 
As noticed in \cite{AvH} (see also Remark \ref{rem AvH1}) we have that
\begin{equation}\label{eq AvH1}
f(D)-f(C)=\beta(D-C)+\int_0^\infty 
(t\,I+C)^{-1}\,( D-C)\,(t\,I+D)^{-1}
\ d\nu(t)\,.
\end{equation}
Now, a closer look at the integrand  reveals that for $t\in (0,\infty)$,
\begin{eqnarray*}
 (t\,I+C)^{-1}\,( D-C)\,(t\,I+D)^{-1}& = & (t\,I+C)^{-1}\,( D-C)\,(t\,I+C+(D-C))^{-1}\\
 & = & (t\,I+C)^{-1}\,( D-C)\,(I+(t\,I+C)^{-1}(D-C))^{-1} \,(t\,I+C)^{-1}\,.
\end{eqnarray*}
Since $\gamma_{tI+C}=t+\gamma_C\geq \gamma_C>0$ then
$\|(tI+C)^{-1}\|\leq \gamma_C^{-1}$, for $t\in (0,\infty)$. Hence, if $\| D-C\|_\fS<\gamma_C$, then  we get that \begin{equation}\label{eq AvH2}
\| (tI+C)^{-1}(D-C)\|_\fS\leq \| (tI+C)^{-1}\|\,\|D-C\|_\fS
\leq \gamma_C^{-1}\,\|D-C\|_\fS <1\quad \text{for} \quad t\in (0,\infty)\,.\end{equation}
Thus, for $n\geq 1$ we have that 
$$
\| \left((tI+C)^{-1}(D-C)\right)^n\|\leq \| \left((tI+C)^{-1}(D-C)\right)^n\|_\fS\leq \| \left((tI+C)^{-1}(D-C)\right)\|_\fS^n\,
$$
since $\|\,\cdot\,\|_\fS$ is submultiplicative. The previous estimates show that the geometric series 
$$
\sum_{n=0}^\infty (-1)^n \left((tI+C)^{-1}(D-C)\right)^n=( I+(tI+C)^{-1}(D-C))^{-1}\in\cG\ell_\fS
$$
is $\|\,\cdot\,\|_\fS$-absolute convergent and $\|\,\cdot\,\|_\fS$-uniformly convergent for $t\in (0,\infty)$, by Eq. \eqref{eq AvH2}. In particular, the series  
is $\|\,\cdot\,\|$-absolute convergent and $\|\,\cdot\,\|$-uniformly convergent for $t\in (0,\infty)$.
Therefore, we now see that 
\begin{eqnarray*}
& &\int_0^\infty \left((t\,I+C)^{-1}-(t\,I+D)^{-1}\right)\ d\nu(t) =\\
 & &\int_0^\infty  (t\,I+C)^{-1}\,( D-C)\,(I+(t\,I+C)^{-1}(D-C))^{-1} \,(t\,I+C)^{-1}\ d\nu(t) =\\
 & &\int_0^\infty  (t\,I+C)^{-1}\,( D-C)\,  \sum_{n=0}^\infty (-1)^n \left((tI+C)^{-1}(D-C)\right)^n \,(t\,I+C)^{-1}\ d\nu(t) =\\
& & \sum_{n=0}^\infty (-1)^n  \int_0^\infty  \left((tI+C)^{-1}\,(D-C)\right)^{n+1} \,(t\,I+C)^{-1}\ d\nu(t) 
=
%\\ & & 
\sum_{n=1}^\infty   \tilde f_{n}(D-C)
\end{eqnarray*}
where, for $n\geq 1$ we let 
$$
\tilde f_{n}(D-C)=(-1)^{n+1}\, \int_0^\infty  \left((tI+C)^{-1}(D-C)\right)^{n} (t\,I+C)^{-1}\ d\nu(t)\in \fS_{sa}
$$ which is a $\|\,\cdot\, \|_\fS$-continuous homogeneous polynomial of degree $n$ (as a function of $D-C$) with values in $\fS_{sa}$, and the series is $\|\,\cdot\, \|_\fS$-absolutely convergent. Indeed, to show $\tilde f_{n}(D-C)\in \fS_{sa}$ we argue as follows: for $n\geq 1$, let $h_n:[0,\infty)\rightarrow \fS_{sa}$ be given by
$$
h_n(t)=\left((tI+C)^{-1}(D-C)\right)^{n} (t\,I+C)^{-1} \in\fS_{sa}\ , \ \ t\geq 0\,.
$$
In this case, $h(t)$ satisfies item \ref{item2} in 
Remark \ref{rem AvH1} since 
\begin{equation}\label{eq ecuac23}
\|h_n(t)\|_\fS\leq (t+\gamma_C)^{-2}\,(\gamma_C^{-1}\,\|D-C\|_\fS)^{n-1}\, \|D-C\|_\fS\ , \ \ t\geq 0\,,
\end{equation}
%with 
$$
(\gamma_C^{-1}\,\|D-C\|_\fS)^{n-1}\, \|D-C\|_\fS \ \int_0^\infty (t+\gamma_C)^{-2}\,\ d\nu(t)<\infty\,,
$$
where we have used that $\int_0^\infty (t^2+1)^{-1}\, d\nu(t)$ is finite. Furthermore, for $\delta>0$ we have that 
$$
\|h(t+\delta)-h(t)\|_\fS
\leq \delta\ 
(\gamma_C^{-1}\,\|D-C\|_\fS)^{n-1}\, \|D-C\|_\fS\ (t+\gamma_C)^{-3} \ \ \text{with} \ \ \int_0^\infty (t+\gamma_C)^{-3} \ d\nu(t)<\infty\,.
$$ Hence, $h(t)$ also satisfies item \ref{item3} in 
Remark \ref{rem AvH1}. By Lemma \ref{lem conv abs} we conclude that 
$$
\tilde f_n(D-C)=(-1)^{n+1}\ \int_0^\infty h(t) \ d\nu(t)\in\fS_{sa}$$ 
and, using Eq. \eqref{eq ecuac23}, we also get that 
\begin{equation}\label{ecuac12}
\|\tilde f_n(D-C)\|_{\fS}\leq (\gamma_C^{-1}\,\|D-C\|_\fS)^{n-1}\, \|D-C\|_\fS \ \int_0^\infty (t+\gamma_C)^{-2}\,\ d\nu(t)\,.
\end{equation}
Hence, $\tilde f_n(D-C)\in\fS_{sa}$ is a $\|\,\cdot\, \|_\fS$-continuous homogeneous polynomial of degree $n\geq 1$. 
Also, Eq. \eqref{ecuac12} shows that the series 
$$
\sum_{n=1}^\infty \tilde f_n(D-C)\in\fS_{sa}
$$
is $\|\,\cdot\, \|_\fS$-absolutely convergent for $\|D-C\|_\fS<\gamma_C$.
The previous facts together with Eq. \eqref{eq AvH1} show that for $D\in\cG\ell^+$ with $\|D-C\|_\fS<\gamma_C$ we have the local expansion 
$$
f(D)=f(C)+\beta(D-C)+\sum_{n=1}^{\infty} \tilde f_{n}(D-C)\,.
$$
 In particular, $f:(C+\fS)\cap \cG\ell^+\rightarrow f(C)+\fS_{sa}$ is a real analytic function.
\end{proof}

\begin{cor}\label{coro sobre AvH derivadas}
Consider the notation in Theorem \ref{teo sobre AvH}. For $D\in (C+\fS)\cap \cG\ell^+$ with 
$\|D-C\|_\fS<\gamma_C$ we have that: 
$$
\|f_1(D-C)\|_\fS\leq (\beta+\int_0^\infty (t+\gamma_C)^{-2}\ d\nu(t))\,\|D-C\|_\fS 
$$ and for $n\geq 2$,
\begin{eqnarray*}
\|f_n(D-C)\|_\fS&\leq &\int_0^\infty (t+\gamma_C)^{-(n+1)}\ d\nu(t)\ \|D-C\|_\fS^n \\
&\leq & \int_0^\infty (t+\gamma_C)^{-2}\ d\nu(t) \ (\|D-C\|_\fS\,\gamma_C^{-1})^{n-1} \,\|D-C\|_\fS
\end{eqnarray*}
\end{cor}
\begin{proof}
The result follows from a direct inspection of the proof of Theorem \ref{teo sobre AvH} above and Lemma \ref{lem conv abs}.
\end{proof}

Integral representations of (Fr${\rm \acute{e}}$chet) derivatives of the functional calculus induced by an operator monotone function have been considered before (see \cite{Card,DMN,Sano}). We remark that our previous result does not only provide integral representations of the derivatives (of all orders) of the functional calculus by operator monotone functions, but it also provides theoretical insights about the relevance of these derivatives for the computation of approximations of the function with respect to symmetrically-normed operator ideals.

\begin{rem}
Consider the notation in Theorem \ref{teo sobre AvH}. The bounds in Corollary \ref{coro sobre AvH derivadas} allow to obtain simple estimates for the norm $\| \, \cdot \, \|_\fS$ of the error in the approximation (remainder)
$f(D)\approx f(C)+ \sum_{n=1}^mf_n(D-C)$, for $m\geq 1$. 
For example, if $f(\la)=\la^{1/2}$ for $\la\in[0,\infty)$ then
Corollary \ref{coro sobre AvH derivadas} implies that 
$$
\|f_n(D-C)\|_\fS \leq \frac{1}{\pi }\int_0^\infty (t+\gamma_C)^{-(n+1)} \,t^{1/2}\ dt\ \|D-C\|_\fS\ \ , \ \ n\geq 1\,,
$$ where we have used the integral representation in Eq. \eqref{eq rep raiz cuad} (so that $\beta=0$). 
It is not difficult to obtain upper bounds for the integral above that in turn allows to obtain upper bounds for the remainder.
Similar results have been obtained in \cite{DMN} for this particular choice of $f$. Nevertheless, notice that the fact that the corresponding functional calculus is real analytic with respect to symmetrically-normed ideals seems to be new even in this case.

There are other important operator monotone functions whose Fr${\rm \acute{e}}$chet derivatives have been considered in the literature. Our results also imply some relevant information about the properties of the corresponding functional calculus and Taylor approximations; we will consider these results elsewhere.
\end{rem}

\subsection{Geometric structure of positive perturbations}

We first give the following (algebraic) characterization of  perturbations of a fixed positive operator by symmetrically-normed ideals. 

\begin{lem}\label{lem sobre la orb posit}
Let $C \in \cC \cR^+$ and $\fS$ be a symmetrically-normed ideal. Then the following conditions are equivalent:
\begin{enumerate}
\item[i)] $D \in \cC \cR^+$ satisfies $D- C \in \fS$ and $[P_{N(D)}:P_{N(C)}]=0$;
\item[ii)] There exists $G \in \cG \ell_\fS$ such that $D=GC G^*$.
\end{enumerate}
\end{lem}
\begin{proof}
\noi $i) \to ii)$   From Remark \ref{ab y moore penrose} we know that $D- C \in \fS$ implies $P_{N(D)}- P_{N(C)} \in \fS$. By the assumption on the essential codimension of these projections and Remark \ref{restricted groups}, we have
$UP_{N(D)}U^*=P_{N(C)}$ for some $U \in \cU_\fS$. Notice that $UDU^*- C \in \fS$ and  $N(UDU^*)=N(C)$; by Corollary \ref{hay cont en p0} we conclude that $UD^{1/2}U^* - C^{1/2} \in \fS$. 
 From the last condition,
and noting that $N(UD^{1/2}U^*)=N(UDU^*)=N(C)=N(C^{1/2})$, it follows  that there exists $G \in \cG \ell_\fS$ such that $GC^{1/2}=UD^{1/2}U^*$ (see Remark \ref{charact close range}). Hence, we find that $(U^*G)C(U^*G)^*=D$, where $U^*G \in \cG \ell_\fS$. 

\medskip

\noi $ii) \to i)$ Clearly, we have $D \in \cC \cR^+$. Using that $G \in \cG \ell_\fS$, it follows that $D-C=GCG^*-C \in \fS$. From $D^*=GCG^*$ we have  $G(N(C)^\perp)=G(R(C^*))=R(D^*)=N(D)^\perp$. By Lemma \ref{smooth unitaries} 
 there is a unitary $U_G \in \cU_\fS$ such that $U_G P_{N(C)} U_G^*=P_{N(D)}$. Then, again by Remark \ref{restricted groups}, we get $[P_{N(D)}: P_{N(C)}]=0$
\end{proof}

For an operator $C \in \cC \cR^+$, since $R(C)^\perp=N(C)$, the three dimensions used in Notation \ref{nota conv inf} reduce to two dimensions. That is, we have 
$\mathbb{J}_C=\{  k \in \Z  :  \, -\dim(N(C)) \leq k \leq \dim(N(C)^\perp) \}$. The following result is an analog of Theorem \ref{index decomp connected} for positive perturbations.

\begin{teo}\label{estru pk real analitica}%\label{teo esp hom posi}
Let $C \in \cC \cR^+$ and $\fS$ be a symmetrically-normed ideal. 
Then 
$$
\cC \cR^+ \cap (C + \fS)=\bigcup_{k \in \mathbb{J}_C} \cP_k(C)
$$
where 
$$
\cP_k(C) =\{  B \in \cC \cR^+   :   B-C \in \fS, \, [P_{N(B)}: P_{N(C)}]=k   \}.
$$ 
The following assertions hold:
\begin{itemize}
\item[i)] The action $\cG \ell_\fS \times \cC \cR^+ \cap (C + \fS) \to   \cC \cR^+ \cap (C + \fS)$,  $G \cdot B=GBG^*$ restricted to $\cP_k(C)$ is well defined and transitive. In other words, 
$\cP_k(C)=\{ GB^{(k)} G^*   : G \in \cG \ell_\fS\}$ for any $B^{(k)} \in \cP_k(C)$.
\item[ii)] $\cP_k(C)$ is a real analytic homogeneous space of $\cG \ell_\fS$ and a  submanifold of $C + \fS$, whose tangent space at $B \in \cP_k(C)$ is 
$(T\cP_k(C))_B=\{ XB + BX^*   :  X \in  \fS \}\subset \fS=(T (C + \fS) )_B$. 
\item[iii)] $\cP_k(C)$ endowed with the previous differential structure is also a submanifold of $\cC_k(C)$.
\end{itemize}
\end{teo}
\begin{proof}
We write $\cP_k=\cP_k(C)$ for short. Analogously to the proof of Theorem \ref{index decomp connected} one can see  that $\cC \cR^+ \cap (C + \fS)$ is decomposed as the above disjoint union, and each $\cP_k \neq \emptyset$ for 
$k \in \mathbb{J}_C$. 

\medskip

\noi $i)$ This follows by elementary properties of the essential codimension.

\medskip

\noi $ii)$ We  only treat the case $k=0$ and take $B^{(0)}=C$. We first show that the isotropy of the action at $C$ given by $\cG_C:=\{  G \in \cG \ell_\fS  :  GCG^*=C  \}$ is a Banach-Lie subgroup of $\cG:=\cG \ell_\fS$. Using a matrix representation with respect to $\cH=R(C)\oplus R(C)^\perp$ we rewrite this group as
$$
\cG_C=\left\{ \begin{pmatrix}  G_{11}   &   G_{12}   \\    0    &   G_{22}  \end{pmatrix}  \in \cG \ell_\fS :   G_{11}C=C(G_{11}^*)^{-1}   \right\},
$$
which has its Lie algebra given by
\begin{align*}
\mathfrak{g}_C & = \{ X \in \fS  :   XC=-CX^*    \} \\
& = \left\{ \begin{pmatrix}  X_{11}   &   X_{12}   \\    0    &   X_{22}  \end{pmatrix}  \in \fS :   X_{11}C=-CX_{11}^*   \right\}\,.
\end{align*}
For every neighborhood $\cT$ of $0  \in \mathfrak{h}$, one can  show that $\exp_\cG(\cT)$ is a neighborhood of $I \in \cG_C$ by using the logarithm series as in Lemma \ref{lie subg}.  To see that 
 $\mathfrak{g}_C$ is complemented in $\mathfrak{g}:=\fS$, we put $P=P_{R(C)}$ and define the closed real subspaces
 $$
 \mathfrak{n}^{\pm}=\{ X \in P\fS P  :  X^*=\pm C^\dagger X C  \}.
 $$
Since every $X \in P \fS P$ can be written as $X=X_+  + X_-$, where $X_+=\frac{1}{2}(X^* + C XC^\dagger)$ and $X_-=\frac{1}{2}(X^* - CXC^\dagger)$, we have 
$P \fS P=\mathfrak{n}^+ \oplus \mathfrak{n}^-$. Therefore, 
$$
\mathfrak{m}=\left\{   \begin{pmatrix}    X_{11}  &   0  \\   Y   & 0  \end{pmatrix} :   X_{11} \in \mathfrak{n}^+ , \, Y \in P^\perp \fS P  \right\}.
$$ 
is a closed supplement of $\mathfrak{g}_C$ in $\fS$. This shows that $\cG_C$ is a Banach-Lie subgroup of $\cG $, and consequently, $\cP_0 \simeq \cG  / \cG_C$ inherits the structure of real analytic homogeneous space of $\cG $ from the quotient space $\cG  / \cG_C$. In particular, the tangent space $(T \cP_0)_C$ can be identified with the quotient space $\mathfrak{g}/\mathfrak{g}_C=\{ XC + CX^*  : X \in \fS  \}$. 

Now we show that $\cP_0$ with its homogeneous structure is also a real submanifold of $C + \fS$. We proceed as in Theorem \ref{estruct variedad y esp homog}.  Notice that the map $\pi_0: \cG \ell_\fS \to \cP_0$, $\pi_0(G)=GCG^*$ admits continuous local cross sections. Indeed, let $B\in \cP_0$,
$P=P_{R(C)}$, $Q=P_{R(B)}$; notice that $Q=Q(B)$ is a continuous function of $B\in \cP_0$ since $Q=BB^\dagger$
and the continuity of the Moore-Penrose inverse on $\cC_0(C)$ shown in Theorem \ref{pseudo conv} (notice that $\cP_0\subset \cC_0(C)$ with the same distance function $d_\fS$). 
 Set $S=QP + (I-Q)(I-P)$ and notice that $SP=QS$ so then $S(I-P)=(I-Q)S$.
It can be checked that  $S$ is invertible when $\|C-B\|_\fS<\gamma$, for sufficiently small $\gamma>0$; in this case, $SPS^{-1}=Q$ and $S(I-P)S^{-1}=(I-Q)$. Then, we define the map
$$
\sigma(B)=B^{1/2} S (C^\dagger)^{1/2} + (I-Q)S(I-P).
$$
If follows by construction that $\sigma(B) \in \cG \ell_\fS$ is such that $ \pi_0(\sigma(B))=B$ for $B\in\cP_0$ such that $\|C-B\|_\fS<\gamma$. 
Moreover, 
$\sigma$ is continuous by the continuity of the square root with respect to the metric $d_\fS$ (see Corollary \ref{teo sobre AvH2}).

It remains to show that the tangent space $(T\cP_0)_C=\{ XC + CX^*  : X \in \fS  \}$ is a closed  complemented subspace in $\fS$. We may use again the above matrix representation, 
$$
XC+CX^*=\begin{pmatrix}  X_{11}C + CX_{11}^*    &    C  X_{21}^*   \\   X_{21}C   &  0   \end{pmatrix}\,.
$$
From this expression, it suffices to show that $\Sigma=\{ XC + CX^*  :   X \in P \fS P   \}$ is closed and complemented in $P \fS P$.
%find a supplement of the subspace $\Sigma=\{ XC + CX^*  :   X \in P \fS P   \}$ in $P \fS P$. 
Consider the bounded and invertible operator  $\cR_{T}: P \fS P \to P \fS P$, 
$\cR_{T}(Y)=YT$, for fixed $T \in \cB(\cH)$ with $R(T)=N(T)^\perp=R(C)$. Since $\cR_{C^\dagger}(\Sigma)=\mathfrak{n}^+$ we conclude that $\Sigma =\cR_{C}(\mathfrak{n}^+)$ is closed and such that 
$\Sigma \oplus \cR_{C}(\mathfrak{n}^-)=P\fS P$. This completes the proof. 

\medskip

\noi $iii)$ Again, we only consider the case $k=0$ and set $\cC_0=\cC_0(C)$. Notice that since both $\cP_0$ and $\cC_0$ are submanifolds of $C+\fS$, so their topologies coincide and the inclusion map  $\iota:\cP_0\to \cC_0$ is real analytic. We now show that this map is an immersion. For it is enough to see  that 
 the image under the tangent map $T_C\iota[(T\cP_0)_C]=(T\cP_0)_C$ is a closed complemented subspace of $(T\cC_0)_C$. 
By inspection of the proof of Theorem \ref{estruct variedad y esp homog} (see Eq. \eqref{eq tang cc0}) we have that 
$$
(T\cC_0)_C =\left\{    \begin{pmatrix}     Z_{11}   &   Z_{12}  \\   Z_{21}   &   0   \end{pmatrix}: Z_{11} \in P \FS P, \, Z_{12} \in P \FS P^\perp, \, Z_{21}  \in  P^\perp \FS P   \right\}\,,
$$ where $P=P_{R(C)}=P_{N(C)^\perp}$ as before. Thus, from the proof of item $ii)$ above we see that 
$$
\left\{    \begin{pmatrix}     Z_{11}   &   0  \\   0   &   0   \end{pmatrix}: Z_{11} \in \cR_{C}(\mathfrak{n}^-)
  \right\}\subset (T\cC_0)_C 
$$ is a complement for $(T\cP_0)_C$ inside $(T\cC_0)_C$. Thus, by \cite[Prop. 8.7]{Up85} we have that $\cP_0\subset \cC_0$ is a real analytic submanifold.
\end{proof}

\begin{rem}
We observe that in the context of $C^*$-algebras the constructions of continuous local cross sections for the action on congruence orbits can be given when $\mathrm{dim}(N(C))< \infty$ or $\mathrm{dim}(R(C))< \infty$  (see \cite[Thm. 3.4]{CMS04}). The fact that the essential codimension is fixed is the key 
for the construction of continuous local cross sections in our previous proof. 
\end{rem}

Now that we have a manifold structure for $\cP_k(C)$ we can complement 
Theorem \ref{teo sobre AvH} by showing the smoothness of the functional calculus by operator monotone functions for the more general domain $\cP_k(C)$. This result will be needed to prove Theorem \ref{teo hay fibras} below.

\begin{cor}\label{teo sobre AvH2}
Let $f:[0,\infty)\rightarrow \R$ be an operator monotone function on $[0,\infty)$, let $C\in\cC\cR^+$ and $\fS$ be a symmetrically-normed ideal. 
Then, 
$f:\cP_k(C) \rightarrow f(C) + \fS_{sa}$
is a  real analytic function.
\end{cor}
\begin{proof}
Without loss of generality we can assume that $k=0$ and check that $f$ is real analytic in a neighborhood of $C\in\cP_0=\cP_0(C)$.
By Theorem \ref{estru pk real analitica} we get that the structure of $\cP_0$ as a submanifold of $C+\fS$ coincides with its structure as an homogeneous space of $\cG\ell$.
%, endowed with the action $ G\cdot C=GCG^*$  for $G\in \cG\ell_\fS$.
 Hence, it suffices to show that the map $\cG\ell_\fS\ni G\mapsto f(GCG^*)$ is real analytic. By Lemma \ref{smooth unitaries}  
there exists a real analytic function $\cG\ell_\fS\ni G\mapsto U_G\in \cU_\fS$ such that $U_G (R(C))=G(R(C))$; thus, the map 
$\cG\ell_\fS\ni G\mapsto U_G^*(GCG^*)U_G$ is real analytic and such that $R(U_G^*(GCG^*)U_G)=R(C)$, for $G\in\cG\ell_\fS$. Hence, we can consider the matrix representation of $U_G^*(GCG^*)U_G$ with respect to the  decomposition $\cH=R(C)\oplus N(C)$ given by
$$
\begin{pmatrix}
U_G^*(GCG^*)U_G|_{R(C)} & 0 \\ 
0 & 0
\end{pmatrix}\,.
$$
Moreover,  $U_G^*(GCG^*)U_G|_{R(C)}\in (C|_{R(C)}+\fS)\cap \cG\ell(R(C))^+$, for $G\in \cG\ell_\fS$. Therefore, by Theorem \ref{teo sobre AvH} we get that 
$f(U_G^*(GCG^*)U_G|_{R(C)}) \in f(C|_{R(C)})+\fS_{sa}$
and that the map
$$
\cG\ell_\fS\ni G\mapsto f(U_G^*(GCG^*)U_G)=
\begin{pmatrix}
f(U_G^*(GCG^*)U_G|_{R(C)}) & 0 \\ 
0 & f(0)
\end{pmatrix}\in f(C) + \fS_{sa} $$
is real analytic, where we have also considered the decomposition $\cH=R(C)\oplus N(C)$ above.
Finally, notice that $\cG\ell_\fS\ni G\mapsto f(GCG^*)= U_G\,f(U_G^*(GCG^*)U_G)\, U_G^*$ is real analytic, because it is the composition of real analytic maps.
\end{proof}

%\subsection{Operator modulus and polar factor}\label{op maodulus and polar factor}
\subsection{Operator modulus and polar factor as real analytic fiber bundles}\label{op maodulus and polar factor}

We apply the previous results 
 to study the maps given by the polar factor and the operator modulus, defined on $\cC_k(A)$, for $A\in\cC\cR$. We remark the well-known fact that the essential codimension is also helpful for the analysis of perturbations of partial isometries by symmetrically-normed ideals;
we collect several results from the literature in the following remark. 

\begin{rem}\label{on connected comp part isom}
For a fixed $V \in \cP \cI\subset \cC\cR$ we consider below the  set $\mathbb{J}_{V}$ previously defined (Notation \ref{nota conv inf}). The essential codimension can be used to write $\cP \cI \cap (V + \fS)$ as the union of connected components:
$$
\cP \cI \cap (V + \fS)=\bigcup_{k \in \mathbb{J}_{V}} \cV_k(V)
$$
where $\cV_k(V)=\{ X \in \cP\cI : X- V \in \fS, \, [P_{N(X)} : P_{N(V)} ]=k  \}\neq \emptyset$. The action  
$
(\cU_\fS \times \cU_\fS) \times \cP \cI \to \cP \cI$, $(U,W) \cdot V=UVW^*$ 
leaves invariant each $\cV_k(V)$, and  moreover, it holds
$$
\cV_k (V)=\{ UV^{(k)}W^*   :    U,W \in  \cU_\fS  \},
$$
where $V^{(k)}$ is any partial isometry in $\cV_k(V)$. These facts were proved in \cite[Prop 3.5]{EC19} for the ideal of Hilbert-Schmidt operators. The same proofs can be carried out for arbitrary symmetrically-normed ideals. In addition, we recall that $\cV_k(V)$ are real analytic  homogeneous spaces of the group $\cU_\fS \times \cU_\fS$, and submanifolds of $V+ \fS$ (see \cite{EC10}).
\end{rem}

\begin{lem}\label{lem sobre a y v}
Let $A \in \cC \cR$ with polar  decomposition $A=V_A|A|$, $k \in \mathbb{J}_A$ and $\fS$ be a symmetrically-normed ideal. The following assertions hold:
\begin{enumerate}
\item[i)] $\alpha: \cC_k(A) \to \cP_k(|A|)$, $\alpha(B)=|B|$, is well defined and surjective.
\item[ii)] $v: \cC_k(A) \to \cV_k(V_A)$, $v(B)=V_B$, is well defined and surjective.
\end{enumerate}
\end{lem}
\begin{proof} Without loss of generality we can assume that $k=0$.

\medskip
\noi $i)$ If $B\in\cC_0(A)$, then $A-B\in\fS$ and $[P_{N(A)}:P_{N(B)}]=0$.
Since $A-B\in\fS$ then $A^*A-B^*B\in\fS$ and since 
$N(A^*A)=N(A)$ and $N(B^*B)=N(B)$ then $B^*B\in \cP_0(A^*A)$.
By Corollary \ref{teo sobre AvH2} we see that $|B|-|A|=(B^*B)^{1/2}-(A^*A)^{1/2}\in\fS$, since
$f(x)=x^{1/2}$ is an operator monotone function on $[0,\infty)$. Since $N(|B|)=N(B^*B)$ and $N(A^*A)=N(|A|)$ we get that 
$\alpha(B)=|B|\in\cP_0(|A|)$. If $C\in\cP_0(|A|)$ then, by item $i)$ in Theorem \ref{estru pk real analitica} we see that there exists $G\in\cG\ell_\fS$ such that $C=G|A|G^*$. By Lemma \ref{smooth unitaries} there exists $U\in\cU_\fS$ such that $U(R(|A|))=G(R(|A|))=R(G|A|)$. If we let $B=(V_AU) C=(V_AU^*) G|A|G^*$, then $V_B=V_AU$ and $|B|=G|A|G^*=C$ are the polar factor and the modulus corresponding to the polar decomposition of $B$. In particular, $V_B-V_A\in \fS$ and $|B|-|A|\in\fS$; then, 
$A-B=V_A|A|-V_B|B|=V_A(|A|-|B|)+(V_A-V_B)|B|\in\fS$ and moreover, $R(B)=R(A)$. Therefore, $B\in \cC_0(A)$ is such 
that $\alpha(B)=C$.

\medskip
\noi $ii)$ If $B\in\cC_0(A)$, then by the first part of the proof, we have that $|B|\in\cP_0(|A|)$. In particular,
$|A|-|B|\in\fS$ and then $|A|^\dagger -|B|^\dagger \in\fS$, by Lemma \ref{ab y moore penrose}.
Hence, $V_A-V_B=A|A|^\dagger-B|B|^\dagger=A(|A|^\dagger -|B|^\dagger )+ (A-B)|B|^\dagger\in\fS$.
On the other hand, observe that $0=[P_{N(A)}:P_{N(B)}]=
%$ implies that $
[P_{N(V_A)}: P_{N(V_B)}]$
%=[P_{N(A)^\perp}:P_{N(B)^\perp}]=0$
 so $v(B)=V_B\in\cV_0(V_A)$.
If $V\in\cV_0(V_A)$ then, by Remark \ref{on connected comp part isom}, we see that there exist $U,\,W\in\cU_\fS$ such that $V=UV_AW$. If we set $B=UV_AW (W^*|A|W)$, then $V_B=UV_AW=V$ and $|B|=W^*|A|W$ are the polar factor and the operator modulus corresponding to the polar decomposition of $B$. 
Since $\cU_\fS\subset \cG\ell_\fS$ then $B=UAW\in \cC_0(A)$ is such that $v(B)=V_B=V$.
\end{proof}

\begin{rem}
%Regarding the continuity of the previous maps, 
We remark that 
%$\alpha: \cC \cR \to \cC \cR^+$ is continuous by the properties of the continuous functional calculus. 
since $v(B)=V_B=B|B|^\dagger$ then, using similar arguments to those considered for the Moore-Penrose map and Corollary \ref{hay cont en p0}, we can show that $\cC_k(A)$ are  maximal subsets (of the metric space $(\cC\cR\cap (A+\fS),d_\fS)$) in which the polar factor becomes a continuous map, for $k \in \mathbb{J}_A$. 
%Similarly, since $|B|=V_B^*B$ we can show that $\cC_k(A)$ are also maximal subsets in which the operator modulus becomes a continuous map, for $k \in \mathbb{J}_A$.
\end{rem}

Under the same notation of Lemma \ref{lem sobre a y v} we now describe the structure of the fibers.

\begin{lem}\label{fibras sub}
Given  $C_0 \in \cP_k(|A|)$ and  $V_0 \in \cV_k(V_A)$, then
\begin{itemize}
\item[i)] $\alpha^{-1}(C_0):=\{ V C_0  \in \cC_k   :    V  \in \cV_k(V_A),  \, N(V)=N(C_0)       \}$; 
 \item[ii)] $v^{-1}(V_0)= \{  V_0C \in \cC_k : C\in \cP_k(|A|),\   N(C)^\perp = R(C)=N(V_0)^\perp \}$.
 \end{itemize}
Furthermore,  both fibers are submanifolds of $A + \fS$. 
\end{lem}
\begin{proof}
We may assume that $k=0$.

\medskip

\noi $i)$ Clearly, we have 
\begin{align}
\alpha^{-1}(C_0) &   =\{ V C_0  \in \cC_0   :    V  \in \cV_0(V_A),  \, N(V)=N(C_0)       \} \nonumber \\
& \simeq \{   V \in \cP \cI   :   V- V_0   \in \fS , \, N(V)=N(V_0)     \}, \label{2ea linea}
\end{align}
where $V_0$ is any partial isometry satisfying  $N(V_0)=N(C_0)$ and $V_0\in\cV_0(V_A)$.
%=UV_AW^*$ for some $U,W \in \cU_\fS$.
 The set in  Eq. \eqref{2ea linea} is a submanifold of $V_A + \fS$ (see \cite[Corol. 3.5]{EC10}). The bijection above is given by $VC_0 \mapsto V$, and it  induces a manifold structure on $\alpha^{-1}(C_0)$. We now prove that $\alpha^{-1}(C_0)$  with the aforementioned manifold structure is  a submanifold of $A+ \fS$. For we first observe that it is not difficult to see $\alpha^{-1}(C_0)$ has the topology defined by the metric $d_\fS$, and the inclusion map  $\iota: \alpha^{-1}(C_0) \to A+ \fS$ is real analytic. To prove that tangent spaces of $\alpha^{-1}(C_0)$ are closed and complemented in $\fS$, we identify the tangent space at $VC_0$ as
$$
(T\alpha^{-1}(C_0))_{VC_0}=\{ XVC_0 :   X \in \fS_{ah} \}.
$$
Without loss of generality, we now assume that $A=VC_0$. We can give the following alternative 
description of the tangent space 
$$
(T\alpha^{-1}(C_0))_A=\left\{  \begin{pmatrix}   Z_{11}  &    0  \\   Z_{21}   & 0   \end{pmatrix}:   Z_{11}A^\dagger  \in P \fS_{ah} P, \, Z_{21} \in P^\perp \fS  Q \right\}\,,
$$
where the above matrix decomposition is in terms of $Q=P_{N(A)^\perp}$ and $P=P_{R(A)}$ like in Eq. \eqref{eq tang cc0}. So we only have to prove that $\Sigma:=\{  Z :  Z A^\dagger \in P\fS_{ah}P \}$ is a closed complemented subspace of $P \fS  Q$. Using the invertible map $\cR_{A^{\dagger}}: P \fS Q \to P \fS P$, $\cR_{A^\dagger}(Z)=ZA^\dagger$, this is equivalent to the fact that $P \fS_{ah}P$ is  a (real) closed complemented subspace of $P\fS P$. Hence the inclusion map $\iota$ is an immersion, and 
$\alpha^{-1}(C_0)$ is a submanifold of $A + \fS$    (see \cite[Prop. 8.7]{Up85}).

\medskip

\noi $ii)$ By Remark \ref{on connected comp part isom} there are unitaries $U,W \in \cU_\fS$ such that $V_0=UV_AW^*$.  If 
$\cS=N(V_0)^\perp$ and $E=P_\cS$,  then  $W^*EW=P_{N(A)^\perp}$. 
Take the positive invertible operator $A_0:=W|A|W^*|_\cS: \cS \to \cS$. Then the fiber can be computed as
\begin{equation}\label{2 da fibra v}
v^{-1}(V_0)  = \{  V_0C \in \cC_0(A) : C\in \cP_0(|A|),\   N(C)^\perp = R(C)=\cS \}
 \simeq \cP_0 (A_0).
\end{equation} 
Here we consider $\cP_0(A_0)$ as a subset of the positive closed range operators acting on $\cS$, and for its definition we  use the ideal $\fS(\cS)=\{  EX|\cS : X \in \fS \}\simeq E\fS E$. The manifold structure of this fiber is induced by the bijection $v^{-1}(V_0) \simeq \cP_0 (A_0)$, which is given by the map $V_0C \mapsto C|_\cS$. Using the characterization in Lemma \ref{lem sobre la orb posit} of $\cP_0(|A|) $ and $\cP_0(A_0)$ one can verify that this map  is well defined, and it has the inverse $C_1 \mapsto V_0(C_1 \oplus 0)$. We follow similar steps to those of the previous item. Again the topology of $v^{-1}(V_0)$ is clearly given by the metric $d_\fS$, and the inclusion map $\iota:v^{-1}(V_0) \to A + \fS$ is real analytic. 
The tangent space at $V_0C$ is given by
$$
(T v^{-1}(V_0))_{V_0C}=\{   V_0   (XC + CX^*) :   X \in E\fS E   \}.
$$
For simplicity, we assume that $V_0 C=A$. This implies that $V_0=V_A$, $C=|A|$, $\cS=N(A)^\perp$
 and $E=Q=P_{N(A)^\perp}$. Next define the continuous invertible map $S_A: P \fS Q  \to Q \fS Q$, 
 $S_A(T)=V_A^* T |A|^\dagger$, which has the following property
 $$
 S_A ((T v^{-1}(V_A))_{A})= \{   X   +   |A| X^*  |A|^\dagger :    X \in Q \fS Q  \}.
 $$
The latter subspace we have already proved to be closed and complemented in $Q \fS Q$ (see the proof of Theorem \ref{estru pk real analitica} $ii)$). These facts show that the inclusion map is an immersion, and hence 
the fiber is a submanifold of $A + \fS$. 
\end{proof}

%Now we give
Next, we state our main result on the operator modulus and polar factor. 

\begin{teo}\label{teo hay fibras} 
Let $A \in \cC \cR$ with polar decomposition $A=V_A|A|$, $k \in \mathbb{J}_A$ and $\fS$ be a symmetrically-normed ideal. 
The following assertions hold:
\begin{enumerate}
\item[i)] $\alpha: \cC_k(A) \to \cP_k(|A|)$, $\alpha(B)=|B|$, is a real analytic fiber bundle.
\item[ii)] $v: \cC_k(A) \to \cV_k(V_A)$, $v(B)=V_B$, is a real analytic fiber bundle.
\end{enumerate}
\end{teo}
\begin{proof}
We may assume again that $k=0$. Set $\cC_0=\cC_0(A)$, $\cP_0=\cP_0(|A|)$ and $\cV_0=\cV_0(V_A)$.

\medskip
\noi $i)$ In this case we show that $\alpha$ is a real analytic function in a neighborhood of $A\in\cC_0(A)$. 
Since $\pi_0:\cG\ell_\fS\times \cG\ell_\fS\to \cC_0$ given by $\pi_0(G,K)=G AK^{-1}$ is a real analytic submersion, then
it is enough to show that $\beta=\alpha\circ \pi_0:\cG\ell_\fS\times \cG\ell_\fS\to \cP_0$ is a real analytic function around the point $(I,I)\in \cG\ell_\fS\times \cG\ell_\fS$.
We show that this is the case by expressing $\beta$ as a composition of real analytic functions.
Set $\cS=N(A)^\perp$ and let $\cG\ell_\fS \ni K\mapsto U_K\in \cU_\fS$ be the real analytic map from 
Lemma \ref{smooth unitaries}, such that $U_K(N(A)^\perp)=K(N(A)^\perp)$.
Then the map $\delta:\cG\ell_\fS\times \cG\ell_\fS\to \cC_0$ given by $\delta(G,K)=\pi_0(G,K)\,U_K=(GAK^{-1})U_K$ is a real analytic map. Notice that $\cG\ell_\fS\times \cG\ell_\fS\ni (G,K)\to |\delta(G,K)|^2$ is a real analytic map taking values in $\cP_0(|A|^2)$. Indeed, $|(GAK^{-1})U_K|^2=(K^{-1}\,U_K)^* |GA|^2 K^{-1}\,U_K$ and $|GA|^2=A^*|G|^2 A$ show that $|GA|^2-|A|^2\in \fS$ and $N(|GA|)=N(|A|)$, so that by Lemma \ref{lem sobre la orb posit}, we get $|(GAK^{-1})U_K|^2 \in\cP_0(|GA|^2)=\cP_0(|A|^2)$.
By Corollary \ref{teo sobre AvH2} applied to the function $f(\lambda)=\lambda^{1/2}$
 we conclude that the map $(G,K)\mapsto |\delta(G,K)|=(U_K^*|GAK^{-1}|^2 U_K)^{1/2}=U_K^*|GAK^{-1}|\,U_K$
is real analytic. Therefore, $\beta(G,K)=\alpha\circ\pi_0(G,K)=U_K |\delta(G,K)|U_K^*$ is a real analytic map. Hence, $\alpha$ is real analytic.

 To prove that $\alpha$ is a real analytic fiber bundle, we also need to show here that $v$ is a real analytic function. To this end notice that $v(B)=V_B=B|B|^\dagger$, for $B\in\cC_0$. On the one hand, $\cC_0\ni B\mapsto |B|\in \cP_0$ is real analytic by the first part of the proof. 
On the other hand, since $\cP_0\subset \cC_0(|A|)$ is a submanifold by Theorem \ref{estru pk real analitica} 
then, $\cP_0(|A|)\ni |B|\mapsto |B|^\dagger$ is a real analytic map, by Theorem \ref{teo pseudo in real analytic}.
Hence,  $v(B)=B|B|^\dagger$ is a real analytic map.

We now show that $\alpha$ is a real analytic fiber bundle. Recall that  the fiber at $C_0\in\cP_0$ is given by $\alpha^{-1}(C_0)  =\{ V C_0  \in \cC_0   :    V  \in \cV_0,  \, N(V)=N(C_0)       \}$.
%\begin{align}
%\alpha^{-1}(C_0) &   =\{ V C_0  \in \cC_0   :    V  \in \cV_0,  \, N(V)=N(C_0)       \} \nonumber \\
%& \simeq \{   V \in \cP \cI   :   V- V_0   \in \fS , \, N(V)=N(V_0)     \}, \label{2era linea}
%\end{align}
%where $V_0$ is any partial isometry satisfying  $N(V_0)=N(C_0)$ and $V_0=UV_AW^*$ for some $U,W \in \cU_\fS$. The set in  Eq. \eqref{2era linea} is a submanifold of $V_A + \fS$ (see \cite[Corol. 3.5]{EC10}).
We consider the action
$\cG\ell_\fS\ni G\mapsto GC_0G^*\in \cP_0$ which, by Theorem \ref{estru pk real analitica}, is a  submersion. Hence, there exist an open set $\cW\subset \cP_0$ with $C_0\in \cW$, and an analytic map (cross-section) $\gamma:\cW\rightarrow \cG\ell_\fS$ such that 
$\gamma(C) \, C_0\,\gamma(C)^*=C$, for $C\in\cW$  (\cite[Corol. 8.3]{Up85}). By Lemma \ref{smooth unitaries} there exists an analytic map
$\cW\ni C\mapsto U_{\gamma(C)}$ such that $U_{\gamma(C)}(R(C_0))=\gamma(C)(R(C_0))=R(C)$. Since $\alpha$ and $v$ are real analytic, and $\alpha^{-1}(C_0)$ is a submanifold of $A+ \fS$ by  Lemma \ref{fibras sub}, the map given by
$$
\alpha^{-1}(\cW)\to \cW\times \alpha^{-1}(C_0)\ , \ \ B=V_B|B|\mapsto (|B|\,,\,\,(V_B \,U_{\gamma(|B|)})\, C_0)
$$
is real analytic, 
and its inverse 
$$\cW\times \alpha^{-1}(C_0)\ni (C,VC_0)\mapsto VU_{\gamma(C)}^*\,C\in \alpha^{-1}(\cW)$$
 is also real analytic. We have used that $\alpha((V_B \,U_{\gamma(|B|)})\, C_0)=C_0$ and $v((V_B \,U_{\gamma(|B|)})\, C_0)=V_B \,U_{\gamma(|B|)}$, for $B\in\alpha^{-1}(\cW)$.

\medskip

\noi $ii)$ By the previous item, we already know that $v$ is a real analytic map. Recall that for $V_0\in \cV_0(V_A)$, we have that $v^{-1}(V_0)  = \{  V_0C \in \cC_0 : C\in \cP_0,\   N(C)^\perp = R(C)=N(V_0)^\perp \}$.
%------------Modificado 14 - 11  --------
%$\cS=N(V_0)^\perp$ and let $P=P_\cS$,  then there is a unitary $\cU_\fS$ such that $UP U^*=P_{N(A)^\perp}$ by Remark \ref{on connected comp part isom}. 
%Take the invertible operator $A_0:=U^*|A|U|_\cS: \cS \to \cS$. Then,
%\begin{equation}\label{2 da fibra v}
%v^{-1}(V_0)  = \{  V_0C \in \cC_0 : C\in \cP_0,\   N(C)^\perp = R(C)=\cS \}
 %\simeq \cP_0 (A_0).
%\end{equation} 
%Here we consider $\cP_0(A_0)$ as a subset of the positive closed range operators on $\cS$, and for its definition we  use the ideal $\fS(\cS)=\{  PX|\cS : X \in \fS \}$. The manifold structure of the fiber is induced by the bijection $v^{-1}(V_0) \simeq \cP_0 (A_0)$, which is given by the map $V_0C \mapsto PC|_\cS$. Using the characterization in Lemma \ref{lem sobre la orb posit} of $\cP_0=\cP_0(|A|) $ and $\cP_0(A_0)$ one can show that this map  is well defined, and it has the inverse $C_1 \mapsto V_0(C_1 \oplus 0)$. 
%----------------Modificado 14-11--------
%for some $C_0\in \cP_0$ such that $N(C_0)^\perp = R(C_0)=\cS$.
%\begin{equation}\label{2 da fibra v}
%v^{-1}(V_0)  = \{  V_0C \in \cC_0 : C\in \cP_0(|A|),\   N(C)^\perp = R(C)=\cS \}
 %\simeq \cP_0 (P|A||_\cS) \subset \cC\cR^+(\cS)\,.
%\end{equation}
%The last identification in Eq. \eqref{2 da fibra v} can be done because one has that $PCP - P|A|P \in \fS$ if and only if $C- |A| \in \fS$, whenever $P-P_{R(|A|)} \in \fS$ and $[P : P_{R(|A|)}]=0$. 
%
%By Theorem \ref{estru pk real analitica}, $CP_0(A_0)$ this is a submanifold of $A_0 + \fS(\cS)$, and consequently, it is also a  submanifold of $U|A|U^* + \fS$.
Consider the action $\cU_\fS\times \cU_\fS\ni (U,W)\mapsto UV_0W\in\cV_0(V_A)$ which, by Remark \ref{on connected comp part isom}, is a  submersion. Then, there exist an open set $\cY\subset \cV_0(V_A)$ with $V_0\in\cY$ and a real analytic cross section 
$\eta:\cY\to \cU_\fS\times \cU_\fS$. So if $\eta(V)=(\eta_1(V),\eta_2(V))$, then
$\eta_1(V) \,V_0\,\eta_2(V)^*=V$, for $V\in\cY$.
Using that $v^{-1}(V_0)$ is a submanifold of $A + \fS$ by  Lemma  \ref{fibras sub}, the map given by
$$
v^{-1}(\cY)\to \cY\times v^{-1}(V_0)\ , \ \ B=V_B|B|\mapsto (V_B,\, V_0 \, 
(\eta_2(V_B)^* \,|B| \,\eta_2(V_B)))
$$
is real analytic with inverse given by 
$$
\cY\times v^{-1}(V_0)\ni (V,B)\mapsto V\,\eta_2(V)\, |B|\,\eta_2(V)^*\in v^{-1}(\cY)
$$
 which is also real analytic.
\end{proof}

\begin{rem} We notice that the maps $\alpha$ and $v$ in Theorem \ref{teo hay fibras} satisfy some compatibility relations with respect to the Moore-Penrose inverse. Indeed,
 consider $A\in \cC\cR$ and $k\in\mathbb J_A=-\mathbb J_{A^*}$. Let $\alpha_{A^*}:\cC_{-k}(A^*)\rightarrow \cP_{-k}(|A^*|)$  and $v_{A^*}:\cC_{-k}(A^*)\rightarrow \cV_{-k}(V_A^*)$ be defined as above, where $A^*=V_A^*|A^*|$ is the polar decomposition of $A^*$. 
We further consider $\alpha_{A^\dagger }:\cC_k(A^\dagger)\rightarrow \cP_k(|A^\dagger|)$ and $v_{A^\dagger }:\cC_k(A^\dagger )\rightarrow \cV_k(V_{A^\dagger})$.
Let $\mu:\cC_k(A)\rightarrow \cC_k(A^\dagger)$ be given by $\mu(B)=B^\dagger$.
Then, notice that for $B\in\cC_k(A)$ we get that
$$
\alpha_{A^\dagger }(\mu(B))=|B^\dagger |=|B^*|^\dagger=\mu(\alpha_{A^*}(B^*)) \ \ , \ \ 
v_{A^\dagger }(\mu(B))=V_B^*=\mu(v_{A^*}(B^*))\,.
$$
\end{rem}

\section{Appendix}\label{apendixity}

In this section we present a proof of Lemma \ref{lem conv abs}.
Hence, we let $\fS$ be a symmetrically-normed ideal, $\nu$ be a positive Borel measure on $(0,\infty)$ and we let $h:[0,\infty)\rightarrow \fS_{sa}$ 
be a function satisfying: 
\begin{enumerate}
\item  $\|h(t)\|_\fS\leq q(t)$, where $q(t)$ is a bounded, continuous, positive and non-increasing function such that  $\int_{0}^{\infty} q(t)\ d\nu(t)<\infty$;
\item For $\delta>0$ then $\| h(t+\delta)-h(t)\|_{\fS}\leq \alpha \ \delta \ r(t)$,  where $\alpha>0$ and $r:[0,\infty)\rightarrow [0,\infty)$ is a  continuous and non-increasing function such that $$\lim_{t\rightarrow \infty}r(t)=0 \ \ \text{and} \ \ \int_0^\infty r(t)\ d\nu(t)<\infty\,.$$
\end{enumerate}
%In this case we show that 
%$$
%\int_0^\infty h(t)\ d\nu(t)\in \fS_{sa} \quad \text{and}\quad \left\|\int_0^\infty h(t)\ d\nu(t)\right\|_{\fS}\leq 
%\int_0^\infty \|h(t)\|_{\fS}\ d\nu(t)\,.
%$$
%Indeed, 
Recall that the operator $\int_0^\infty h(t)\ d\nu(t)\in \cB(\cH)_{sa}$ is determined by the identity
$$
\langle \int_0^\infty h(t)\ d\nu(t) \ x \, , \, y\rangle =\int_0^\infty \langle h(t)\ x \, , \, y \rangle \ d\nu(t)\quad , \quad x \, , \ y\in \cH\,.
$$
Since $\|\,\cdot\,\|\leq \|\,\cdot\,\|_{\fS}$ then item \ref{item2} above shows that $\int_0^\infty \|h(t)\|\ d\nu(t)<\infty$. Then, the previous facts imply that for every $p\geq 1$, 
$$
\left\|\int_0^\infty h(t)\ d\nu(t)\right\|\leq \int_0^\infty \|h(t)\|\ d\nu(t) \quad , \quad \int_0^\infty h(t)\ d\nu(t)=
\sum_{m=1}^\infty \int_{\left [0\, , \, \frac{1}{2^p}\right )} h(\frac{m-1}{2^p}+t)\ d\nu(t)
\,,
$$ where the series converges in the operator norm. We now introduce the sequence $$R_p=\sum_{m=1}^\infty \nu\left ([\frac {m-1}{2^p}\, , \, \frac{m}{2^p})\right)\ h(\frac{m}{2^p})\in \fS_{sa}\,, \ p\geq 1\,,$$
where we have used that $\|h(t)\|_{\fS}\leq q(t)$ is a decreasing function, $$
\sum_{m=1}^\infty \nu\left ([\frac {m-1}{2^p}\, , \, \frac{m}{2^p})\right)\ \|h(\frac{m}{2^p})\|_\fS\leq 
\sum_{m=1}^\infty \nu\left ([\frac {m-1}{2^p}\, , \, \frac{m}{2^p})\right)\ q(\frac{m}{2^p})\leq 
\int_0^\infty q(t)\ d\nu(t)<\infty\,,$$ so that the series defining $R_p$ is absolutely convergent in $\fS$ (and hence determines an element in $\fS_{sa}$).
By item \ref{item3} above we get that for $m,\,p\geq 1$:
\begin{equation}\label{eq relac normis1}
\left \| h(\frac{m}{2^p})- h(\frac{m-1}{2^p}+t)\right \|\leq \left \| h(\frac{m}{2^p})-h(\frac{m-1}{2^p}+t)\right \|_{\fS}\leq \frac{\alpha}{2^p} \ r(\frac{m-1}{2^p}) \, , \ \ \text{for} \ \ t\in [0,\frac{1}{2^p})\,,\end{equation} where we have used that 
$ \frac{m}{2^p}=(\frac{m-1}{2^p}+t)+\delta$, for some $\delta\in [0,2^{-p})$, and that 
$r(\frac{m-1}{2^p}+t)\leq r(\frac{m-1}{2^p})$, for $t\in [0,2^{-p})$.
By the same item we also get that there exists $\eta\geq 1$ such that for $p\geq 1$ we have that 
\begin{equation}\label{comparac sum r int1}
\sum_{m=1}^\infty r(\frac{m-1}{2^p})\ \nu([\frac {m-1}{2^p}\, , \, \frac{m}{2^p}))\leq \eta\,\int_0^\infty r(t)\, d\nu(t)\,.
\end{equation}
Therefore, we get that 
\begin{equation}\label{eq hay conv en norm de rps}
\left \| R_p-\int_0^\infty h(t)\ d\nu(t) \right\|\leq \frac{\alpha\cdot \eta}{2^p}\int_0^\infty r(t)\ d\nu(t) \xrightarrow[p\rightarrow \infty]{} 0\,.
\eeq
We now show that $\{R_p\}_{p\geq 1}$ is a Cauchy sequence in $\fS$: indeed, for $p\geq 1$ we have that 
$$
R_p=\sum_{m=1}^\infty \left(\nu\left ([\frac {2m-2}{2^{p+1}}\, , \, \frac{2m-1}{2^{p+1}})\right) +\nu\left ([\frac {2m-1}{2^{p+1}}\, , \, \frac{2m}{2^{p+1}})\right) \right) \ h(\frac{2m}{2^{p+1}})
$$
and hence,
$$
R_{p+1}-R_{p}= 
\sum_{m=1}^\infty \nu\left ([\frac {2m-2}{2^{p+1}}\, , \, \frac{2m-1}{2^{p+1}})\right) \left ( h(\frac{2m-1}{2^{p+1}}) - h(\frac{2m}{2^{p+1}})\right)\,.
$$The previous identity implies that 
$$
\|R_{p+1}-R_p\|_{\fS}\leq 
\sum_{m=1}^\infty \nu\left ([\frac {2m-2}{2^{p+1}}\, , \, \frac{2m-1}{2^{p+1}})\right) \  \frac{\alpha}{2^{p+1}}\ 
 r(\frac{2m-1}{2^{p+1}})\leq \frac{\alpha\cdot \eta}{2^{p+1}} \ \int_0^\infty r(t)\ d\nu(t)\,.
$$ This last fact shows that $\{R_p\}_{p\geq 1}$ is a Cauchy sequence in $\fS$ so that it converges to an operator
$R\in\fS_{sa}$; but then, $\{R_p\}_{p\geq 1}$ also converges to $R$ in the operator norm. The previous comments together with Eq. \eqref{eq hay conv en norm de rps} imply that $R=\int_0^\infty h(t)\ d\nu(t)\in \fS_{sa}$. 
On the other hand, using Eq. \eqref{eq relac normis1} we have that
$$
\left | \|h(\frac{m-1}{2^p}+\delta)\|_\fS-\|h(\frac{m}{2^p})\|_\fS\right |\leq \left \| h(\frac{m-1}{2^p}+\delta)-h(\frac{m}{2^p})\right \|_{\fS}\leq \frac{\alpha}{2^p} \ r(\frac{m-1}{2^p}) \, , \ \ \text{for} \ \ \delta\in [0,\frac{1}{2^p})\,.
$$
Hence, using the previous fact and Eq. \eqref{comparac sum r int1} we get that 
\begin{equation}\label{eq la ultima eq1}
\left | 
\int_0^\infty \|h(t)\|\ d\nu(t)-\sum_{m=1}^\infty \nu\left ([\frac {m-1}{2^p}\, , \, \frac{m}{2^p})\right )\ \|h(\frac{m}{2^p})\|_\fS
\right |\leq \frac{\alpha\cdot \eta}{2^p}\ \int_0^\infty r(t)\ d\nu(t)\,.
\end{equation} Since $\|R_p-\int_0^\infty h(t)\ d\nu(t)\|_\fS\rightarrow 0$ when $p\rightarrow \infty$ then, given $\epsilon>0$ there exists $p_0\geq 1$ such that if $p\geq p_0$ then 
$\|\int_0^\infty h(t)\ d\nu(t)\|_\fS\leq \|R_p\|_\fS+\varepsilon$. On the other hand, Eq. \eqref{eq la ultima eq1}
shows that there exists $p_1\geq p_0$ such that for $p\geq p_1$ we get that 
$$
\|\int_0^\infty h(t)\ d\nu(t)\|_\fS\leq \|R_p\|_\fS+\varepsilon \leq \sum_{m=1}^\infty \nu\left ([\frac {m-1}{2^p}\, , \, \frac{m}{2^p})\right)\ \|h(\frac{m}{2^p})\|_\fS+\varepsilon\leq 
\int_0^\infty \|h(t)\|_\fS\ d\nu(t)+2\,\varepsilon\,.
$$

\subsection*{Acknowledgment}
This research was partially supported by  CONICET  (PIP 2021/2023 11220200103209CO),  ANPCyT (2015 1505/ 2017 0883) and FCE-UNLP (11X974).

\bigskip

{\sc (Eduardo Chiumiento)} {Departamento de  Matem\'atica \& Centro de Matem\'atica La Plata, FCE-UNLP, Calles 50 y 115, 
(1900) La Plata, Argentina  and Instituto Argentino de Matem\'atica, `Alberto P. Calder\'on', CONICET, Saavedra 15, 3er. piso,
(1083) Buenos Aires, Argentina.}     
                                               
\noi e-mail: {\sf eduardo@mate.unlp.edu.ar}   

\bigskip

{\sc (Pedro Massey)} {Departamento de  Matem\'atica \& Centro de Matem\'atica La Plata, FCE-UNLP, Calles 50 y 115, 
(1900) La Plata, Argentina  and Instituto Argentino de Mate\-m\'atica, `Alberto P. Calder\'on', CONICET, Saavedra 15, 3er. piso,
(1083) Buenos Aires, Argentina.}

\noi e-mail: {\sf massey@mate.unlp.edu.ar}


{\small

\begin{thebibliography}{99}



\bibitem{AS94} W.O. Amrein, K.B. Sinha,  On pairs of projections in a Hilbert space, Linear Algebra Appl. 208/209 (1994), 425-435.

\bibitem{A14}  E. Andruchow, Pairs of Projections: Geodesics, Fredholm and Compact Pairs. Complex Anal. Oper. Theory 8 (2014), 1435-1453.


\bibitem{ACL18} E. Andruchow, E. Chiumiento, G. Larotonda, Geometric significance of Toeplitz kernels, J. Funct. Anal. (2018), 329-355.

\bibitem{ACM05} E. Andruchow, G. Corach, M. Mbekhta, On the geometry of generalized inverses, Math. Nachr. 278 (2005), 756-770.

\bibitem{AL08} E. Andruchow, G. Larotonda, Hopf-Rinow theorem in the Sato Grassmannian. J. Funct. Anal. 255 (2008), no. 7, 1692-1712.

\bibitem{ALR10} E. Andruchow, G. Larotonda, L. Recht,  Finsler geometry and actions of the p-Schatten   	unitary groups, Trans. Amer. Math. Soc. 362 (2010), 319-344.


\bibitem{AvH}  T. Ando, J.L. van Hemmen,  An inequality for trace ideals, Comm. Math. Phys. 76 (1980), 143-148.

%\bibitem{AL08} E. Andruchow, G. Larotonda, Hopf-Rinow theorem in the Sato Grassmanian, J. Funct. Anal. 255  (2008), no. 7, 1692-1712.

%\bibitem{AL10} E. Andruchow, G. Larotonda, The rectifiable distance in the unitary Fredholm group, Studia Math. 196 (2010), 151-177. 
\bibitem{CAp} C. Apostol, The reduced minimum modulus. Michigan Math. J. 32 (1985), no. 3, 279-294.

\bibitem{ASS94} J. Avron, R. Seiler, B. Simon, The index of a pair of projections, J. Funct. Anal. 120 (1994), no. 1, 220-237. 

\bibitem{B} D. Belti\c t$\breve{\text{a}}$,  Smooth homogeneous structures in operator theory, Chapman and Hall/CRC, Monographs and Surveys in Pure and Applied Mathematics 137, Boca Raton, 2006.


\bibitem{BL23} D. Belti\c t$\breve{\text{a}}$, G. Larotonda,  Unitary Group Orbits Versus Groupoid Orbits of Normal Operators, J. Geom. Anal.  33 (2023), 95.

%\bibitem{BPW16}   D. Belti\c{t}$\breve{\text{a}}$, S. Patnaik, G. Weiss, Cartan subalgebras of operator ideals, Indiana Univ. Math. J. 65 (2016), no. 1, 1-37. 




\bibitem{BGJP} D. Belti\c t$\breve{\text{a}}$, T. Goli${\rm \acute{n}}$ski, G. Jakimowicz, F. Pelletier, 
Banach-Lie groupoids and generalized inversion. J. Funct. Anal. 276 (2019), no. 5, 1528-1574.

\bibitem{BR05} D. Belti\c{t}$\breve{\text{a}}$, T.S. Ratiu, Symplectic leaves in real Banach Lie-Poisson spaces. Geom. Funct.
Anal. 15 (2005), no. 4, 753-779.



\bibitem{Boa} E. Boasso, On the Moore-Penrose inverse in $C^\ast$-algebras. Extracta Math. 21 (2006), no. 2, 93-106.



%\bibitem{BDF73} L. G. Brown, R. G. Douglas, and P. A. Fillmore, Unitary equivalence modulo
%the compact operators and extensions of
%C$^*$-algebras, Lecture Notes in
%Mathematics vol. 345 (Apr. 13, 1973), ed. by P. A. Fillmore, pp. 58-128.


%\bibitem{BG} A. Ben-Israel and T.N.E. Greville, Generalized inverses: theory and applications, Robert E. Krieger Publishing Co., Inc., Huntington, N.Y., 1980.

%\bibitem{BPW16}   D. Belti\c{t}$\breve{\text{a}}$, S. Patnaik, G. Weiss, Cartan subalgebras of operator ideals, Indiana Univ. Math. J. 65 (2016), no. 1, 1-37. 

\bibitem{BDF73} L.G. Brown, R.G. Douglas, and P. A. Fillmore, Unitary equivalence modulo the compact operators and extensions of C$^*$-algebras, in: P. A. Fillmore (Ed.) Lect. Not. Math. 345 (Apr. 13, 1973)  58-128.

\bibitem{Card} J.R. Cardoso, Computation of the matrix $p$-th root and its Fr${\rm \acute{e}}$chet derivative by integrals. Electron. Trans. Numer. Anal. 39 (2012), 414-436.

\bibitem{C85} A.L.Carey, Some homogeneous spaces and representations of the Hilbert Lie group $U(H)_2$, Rev. Roum. Math. Pures Appl. 30 (1985), no. 7, 505-520.

\bibitem{CWS96}    G. Chen, M. Wei, Y. Xue,  Perturbation Analysis of the
Least Squares Solution in Hilbert Spaces, Linear Algebra Appl. 244 (1996), 69-80.

\bibitem{EC10} E. Chiumiento, Geometry of $\cJ$-Stiefel manifolds, Proc. Amer. Math. Soc. 138 (2010), no. 1, 341-353.

\bibitem{EC19} E.  Chiumiento, Global symmetric approximation of frames, J. Fourier Anal. Appl. 25 (2019), no. 4, 1395-1423.

\bibitem{ECPM1} E. Chiumiento, P. Massey, On restricted diagonalization, J. Funct. Anal. 282 (2021), 109342.

\bibitem{ECPM23} E. Chiumiento, P. Massey, Restricted orbits of closed range operators and equivalences of frames for subspaces, submitted. 


\bibitem{DMN} P. Del Moral, A. Niclas, A Taylor expansion of the square root matrix function. J. Math. Anal. Appl. 465 (2018), no. 1, 
259-266.



%\bibitem{C} J.B. Conway,  A Course in Functional Analysis, 2nd edn. Springer-Verlag, New York (1990). 

\bibitem{CMM09} G. Corach, A. Maestripieri, M. Mbekhta, Metric and homogeneous structure of closed range operators, J. Operator Theory 61 (2009), no. 1, 171-190.


\bibitem{CMS04} G. Corach, A. Maestripieri, D. Stojanoff, Orbits of positive operators from a differentiable viewpoint, Positivity 8 (2004), no. 1, 31-48.




\bibitem{dH72} P. de la Harpe, Classical Banach-Lie Algebras and Banach-Lie
Groups of Operators in Hilbert Space, Lecture Notes in Mathematics,
Vol. 285. Springer-Verlag, Berlin-New York, 1972.


\bibitem{Dyk04} K. Dykema, T. Figiel, G. Weiss, M. Wodzicki, Commutator structure of operator ideals, Adv. Math. 185 (2004), no. 1, 1-79. 

\bibitem{GPe73} G.H. Golub, V. Pereyra, The differentiation of pseudo-inverses and nonlinear least squares problems whose variables separate. SIAM J. Numer. Anal. 10 (1973), 413-432.


\bibitem{GP91} M.C. Gouveia, R. Puystjens, About the group inverse and Moore-Penrose inverse of a product, Linear 
Algebra Appl. 150 (1991), 361-369. 

\bibitem{I83} S. Izumino, Convergence of generalized splines and splines projections, J. Approx. Theory 38 (1983), 269-278.



\bibitem{BJia} B. Jia, Compact perturbations of Moore-Penrose invertible operators. J. Math. Anal. Appl. 523 (2023), no. 2, Paper No. 127048, 10 pp.

\bibitem{KL17} V. Kaftal, J. Loreaux, Kadison's Pythagorean theorem and essential codimension, Integral Equations Operator Theory 87 (2017), no. 4, 565-580.




\bibitem{K01} J.J. Koliha, Continuity and differentiability of the Moore-Penrose inverse in C$^*$-algebras, Math. Scand. 88 (2001), 154-160.

\bibitem{LMb92} J.-Ph. Labrousse, M. Mbekhta, Point-of-continuity operators for the conorm and the Moore-Penrose inverse. Houston J. Math. 18 (1992), no. 1, 7-23.


\bibitem{Lar19} G. Larotonda, Metric geometry of infinite dimensional Lie groups and their homogeneous spaces, Forum Math. 31 (2019) 1567-1605.


\bibitem{LR} J. Leiterer, L. Rodman, Smoothness of generalized inverses. Indag. Math. (N.S.) 23 (2012), no. 3, 487-521.

\bibitem{L19} J. Loreaux, Restricted diagonalization of finite spectrum normal operators and a theorem of Arveson, J. Operator Theory 81 (2019), no.  2,  257-272.


\bibitem{GK60} I.C. Gohberg, M.G. Krein, Introduction to the Theory of Linear Non-Self-Adjoint Operators, Amer. Math. Soc., Providence, RI, 1960.


%\bibitem{LR12} J. Leiterer, L. Rodman, Smoothness of generalized inverses, Indag. Math. 23 (2012), 487-521.

%\bibitem{PS} A. Pressley, G. Segal, Loop Groups, Oxford Math. Monogr., Clarendon/Oxford Univ. 
%Press, Oxford, 1990.

\bibitem{neeb98} K.-H.Neeb, Holomorphic highest weight representations of infinite dimensional
complex classical groups, J. Reine Angew. Math. 497 (1998), 171-222.

\bibitem{N00} K.-H. Neeb, Infinite-dimensional groups and their representations. Lectures at the European School in Group Theory, SDU-Odense Univ., August 2000.  





\bibitem{Sano} T. Sano, Fr${\rm \acute{e}}$chet derivatives for operator monotone functions. Linear Algebra Appl. 456 (2014), 88-92.

\bibitem{S79}  B. Simon, Trace Ideals and their Applications, in: London Mathematical Society Lecture Note Series, vol. 35, Cambridge University Press, Cambridge,
1979.

\bibitem{S17} B. Simon, Unitaries permuting two orthogonal projections, Linear Algebra Appl. 528 (2017), 436-441. 

\bibitem{SV78} S. Str$\breve{\text{a}}$til$\breve{\text{a}}$, D. Voiculescu, On a class of KMS states for the group $U(\infty)$, Math. Ann. 235
(1978), 87-110.



\bibitem{S77} G.W. Stewart, On the perturbation of pseudo-inverses, projections, and linear least squares problems, SIAM Rev. 19 (1977), 634-662.


\bibitem{Up85} H. Upmeier,  Symmetric Banach Manifolds and Jordan C$^*$-Algebras, North-Holland
Math. Stud., vol. 104, Notas de Matem\'atica, vol. 96, North-Holland, Amsterdam 1985.


\bibitem{W73} P. \r{A}. Wedin, Perturbation theory for pseudo-inverses, BIT 13 (1973),  217-232.


\end{thebibliography}}
\end{document}